\documentclass[12pt]{amsart}
\usepackage{amssymb,latexsym,amsmath,epsfig, bbm}
\usepackage[hyphens]{url} 
\usepackage[colorlinks,linkcolor=blue]{hyperref}
\usepackage[utf8]{inputenc}
\usepackage[T1]{fontenc}
\usepackage{graphicx}
\usepackage{mathrsfs}

\hypersetup{colorlinks = true,linkcolor = blue,anchorcolor =red,citecolor = blue,filecolor = red,urlcolor = red,pdfauthor=author}

\def\pf{\begin{proof}}

	\usepackage{amssymb,textcomp, url}
	
	\usepackage{geometry}

	\newtheorem{thm}{Theorem}[section]
	\newtheorem{prop}[thm]{Proposition}
	\newtheorem{lem}[thm]{Lemma}
	\newtheorem{cor}[thm]{Corollary}
	\theoremstyle{definition}
	
	\newtheorem{defn}[thm]{Definition}

	\newtheorem{rem}[thm]{Remark}
	\numberwithin{equation}{section}
	
	\usepackage{t1enc,mathrsfs,amssymb,amsfonts,latexsym,pifont,fancybox}
	\newcommand{\bprop} {\begin{proposition}}
		\newcommand{\eprop} {\end{proposition}}
	\newcommand{\btheo} {\begin{theorem}}
		\newcommand{\etheo} {\end{theorem}}
	\newcommand{\blem} {\begin{lemma}}
		\newcommand{\elem} {\end{lemma}}
	\newcommand{\bcor} {\begin{corollary}}
		\newcommand{\ecor} {\end{corollary}}
	
	\newcounter{rea}
	\newcounter{red}
	\newcommand{\Be}{\begin{equation}}
		\newcommand{\Ee}{\end{equation}}
	\newcommand{\Bea}{\begin{eqnarray}}
		\newcommand{\Eea}{\end{eqnarray}}
	\newcommand{\Bes}{\begin{equation*}}
		\newcommand{\Ees}{\end{equation*}}
	\newcommand{\Beas}{\begin{eqnarray*}}
		\newcommand{\Eeas}{\end{eqnarray*}}
	\newcommand{\Ba}{\begin{array}}
		\newcommand{\Ea}{\end{array}}

	\scrollmode
	
	\scrollmode

	\title[Weakly, sufficiently or strongly localized operators] {Weakly, sufficiently or strongly localized operators on the Fock space in $\mathbb C^n$}



\begin{document}

		\begin{abstract}
We study properties of the following four classes of operators on the Fock space in $\mathbb C^n:$ 1) weakly localized operators; 2) sufficiently localized operators in the sense of Xia and Zheng; 3) sufficiently localized operators; 4) strongly localized operators. In this respect, we examine composition operators, Toeplitz operators with a measure symbol whose total variation measure is a Fock-Carleson measure, and singular operators of convolution type introduced by Zhu, among others. We also provide a bounded operator which is not weakly localized and does not even belong to the Toeplitz algebra. 

Class 1) contains class 2), class 2) contains class 3), which clearly contains class 4). We prove that the first two inclusions are strict. Our proofs are in terms of singular operators of convolution type introduced by Zhu. The third inclusion was already known to be strict, as Wang, Cao and Zhu exhibited examples of composition operators which are sufficiently localized, but are not strongly localized. 
		\end{abstract}

			\author[D. B\'ekoll\`e]{David B\'ekoll\`e}
			\address{Department of Mathematics, Faculty of Science, University of Yaound\'e I, P.O. Box 812, Yaound\'e, Cameroon }
			\email{{\tt david.bekolle@univ-yaounde1.cm \& dbekolle@gmail.com}}
			\author[H. O. D\'efo]{Hugues Olivier D\'efo}
			\address{Department of Mathematics, Faculty of Science, University of Yaound\'e I, P.O. Box 812, Yaound\'e, Cameroon }
			\email{{\tt hugues.defo@facsciences-uy1.cm \& deffohugues@gmail.com}}
			\author[S. B. Difo]{Solange Bridgitte Difo}
			\address{AIMS-Cameroon, Crystal Gardens, P.O. Box 608, Limbe, Cameroon}
			\email{{\tt solange.difo@aims.cameroon.org}}
			\author[E. L. Tchoundja]{Edgar Landry Tchoundja}
			\address{Department of Mathematics, Faculty of Science, University of Yaound\'e I, P.O. Box 812, Yaound\'e, Cameroon }
			\email{\tt edgartchoundja@gmail.com}
		
		\maketitle

		\section{Introduction}\label{sec1}
Let $n$ be a positive integer and denote by $dV$ the standard volume measure in $\mathbb C^n.$ For $p>0,$ let $L^p_1 (\mathbb C^n, dV)$ be the Lebesgue space of measurable functions $f$ on $\mathbb C^n$ such that
$$\left \Vert f\right \Vert_{L_1^p}^p=\left (\frac p{2\pi} \right )^n\int_{\mathbb C^n} \left \vert f(z)e^{-\frac {|z|^2}2}\right \vert^pdV(z)<\infty.$$
Furthermore, for $p=\infty,$ we denote by $L^\infty_1 (\mathbb C^n, dV)$ be the Lebesgue space of measurable functions $f$ in $\mathbb C^n$ such that
$$\left \Vert f\right \Vert_{L_1^\infty}={ess\, sup}{\left \{\left \vert f(z)e^{-\frac {|z|^2}2}\right \vert: z\in \mathbb C^n\right \}}<\infty.$$
The classical Fock space $F^p_1$ is the space of entire functions on $\mathbb C^n$ which belong to $L^p_1 (\mathbb C^n, dV).$ Similarly, the 
Fock space $F^\infty_1$ is the space of entire functions on $\mathbb C^n$ which belong to $L^\infty_1 (\mathbb C^n, dV).$\\
Let $d\mu$ be the Gaussian measure in $\mathbb C^n.$ In terms of the standard volume measure $dV$ on $\mathbb C^n,$ it is given by
$$d\mu (z)=\pi^{-n}e^{-|z|^2}dV(z).$$
The Fock space $H^2 (\mathbb C^n, d\mu)$ is defined to be the subspace of the (Hilbert-)Lebesgue space $L^2 (\mathbb C^n, d\mu)$ consisting of entire functions. Clearly, $H^2 (\mathbb C^n, d\mu)=F^2_1.$ The symbol $K_z$ denotes the reproducing kernel and the symbol $k_z$ denotes the normalized reproducing kernel for $H^2 (\mathbb C^n, d\mu).$ That is,
$$K_z (\zeta)=e^{\langle \zeta, z\rangle}, \quad k_z (\zeta)=e^{\langle \zeta, z\rangle}e^{-\frac {|z|^2}2},\qquad z, \zeta \in \mathbb C^n.$$
The inner product on $H^2 (\mathbb C^n, d\mu)$ is denoted by $\langle \cdot, \cdot \rangle$ and its associated norm is denoted by $\left \vert \cdot\right  \vert.$ It holds: $\left \Vert K_z\right \Vert=e^{\frac {|z|^2}2}.$\\
The Fock spaces and linear operators defined on them have been studied by many authors and the basic reference for $n=1$ is \cite{Z} 
 In the sequel, we shall denote by $\mathcal L(H^2 (\mathbb C^n, d\mu))$ the space of all bounded linear operators on $H^2 (\mathbb C^n, d\mu).$ For  $T\in \mathcal L(H^2 (\mathbb C^n, d\mu)),$ the Berezin transform $\widetilde T: \mathbb C^n\rightarrow \mathbb C$ is the function defined by 
$$\widetilde T(z)=\langle Tk_z, k_z\rangle.$$
Let $f$ be a bounded measurable function in $\mathbb C^n.$  The Toeplitz operator $T_f$ with symbol $f$ is defined as follows.
$$T_f g(z)=\int_{\mathbb C^n} f(w)g(w)\overline{K_z (w)}d\mu (w), \quad z\in \mathbb C^n, \quad g\in H^2 (\mathbb C^n, d\mu).$$
In the sequel, a Toeplitz operator $T_f$ with bounded measurable symbol $f$ in $\mathbb C^n$ will simply be called {\textit{Toeplitz operator}}. We denote by $\mathcal T^{1}$ the operator-norm closure of the linear span of the Toeplitz operators.  In fact, $\mathcal T^{1}$ coincides with the Toeplitz algebra, which is the $C^*$-algebra generated by the Toeplitz operators. 
In 2012, Bauer and Isralowitz \cite{BI12}
proved the following theorem.
\begin{thm}\label{thm11}
If $T\in \mathcal L(H^2 (\mathbb C^n, d\mu)),$  then $T$ is compact if and only if $T\in \mathcal T^{1}$ and its Berezin transform $\widetilde T$ vanishes at infinity.
\end{thm}
We quote \cite{IMW15}
: "However, it is in general very difficult to check whether a given operator $T$ is in the Toeplitz algebra, unless $T$ is itself a Toeplitz operator or a combination of a few Toeplitz operators, and as such, one would like a 'simpler" sufficient condition to guarantee this." As a 'simpler" sufficient condition, these authors \cite{IMW15}
 propose the \textit{weak localized} condition.

Wang, Cao and Zhu \cite{WCZ13}
introduced the first notion of sufficiently localized operators. In their Section 5, they proved  that every Toeplitz operator is also sufficiently localized in their own sense. Xia and Zheng \cite{XZ13}
 introduced another notion of sufficiently localized operators. In their paper, they gave several examples of such operators and they prove that an operator from this class is compact if and only if its Berezin transform vanishes at infinity. In particular, they showed that the $C^*$-algebra generated by their class of sufficiently localized operators contains the Toeplitz algebra. This follows from the following result: every Toeplitz operator is sufficiently localized in their sense \cite[Proposition 4.1]{XZ13}. Earlier, Wang, Gao and Zhu \cite[Section 5]{WCZ13} proved for $n=1$ that every Toeplitz operator is also sufficiently localized in their own sense.

The general setting for Bargmann-Fock spaces is the following. Let $d^c=\frac i4\left (\overline{\partial} -\partial \right )$ and let $d$ be the usual exterior derivative; the real-valued function $\phi\in C^2 (\mathbb C^n)$ is such that $c\omega_0 <dd^c \phi <C\omega_0$ holds uniformly pointwise in $\mathbb C^n$ for some positive constants $c$ and $C$ (in the sense of positive $(1, 1)$ forms), where $\omega_0= dd^c \left \vert \cdot \right \vert^2$ is the standard Euclidean K\"ahler form. The generalized Bargmann-Fock space $F^p_\phi, 0<p<\infty$ is the space of entire functions in $\mathbb C^n$ such that $fe^{-\phi}\in L^p (\mathbb C^n, dV).$ 
The Bargmann-Fock space $F^2_\varphi$ is a reproducing kernel Hilbert space and we denote by $K_z^\phi$ its reproducing kernel. The Toeplitz kernel $T_f^\phi$ with symbol $f\in L^\infty (\mathbb C^n)$ is defined as follows.
$$T_f^\phi g(z)=\int_{\mathbb C^n} f(w)g(w)\overline{K_z^\phi (w)}e^{-2\phi (w)}dV(w), \quad z\in \mathbb C^n, \quad f\in F^2_\varphi.$$

In \cite{IMW15}
, Isralowitz, Mitkovski and Wick introduced the notion of weakly localized operators defined on the generalized Bargmann-Fock space $F^2_\phi$ and they showed in their Proposition 3.2 that every Toeplitz operator is weakly localized. 
These authors also established in terms of the vanishing at infinity of their corresponding Berezin transforms, 
a sufficient condition for the compactness of weakly localized operators on $F^p_\phi, \hskip 1truemm 1<p<\infty$ \cite[Theorem 1.3]{IMW15} 

In the sequel, we restrict to the case where $\phi (z)=\frac {\alpha|z|^2}2$ and to their corresponding Lebesgue spaces $L^p_\alpha$ and  Fock spaces $F^p_\alpha,$ with $p, \alpha>0.$

Finally, Xia \cite{X15}
 outlined the proof of the following result.
\begin{thm}\label{thm12}
The $C^*$-algebra generated by the class of weakly localized operators on $H^2 (\mathbb C^n, d\mu)$ coincides with $\mathcal T^1.$
\end{thm}

In his paper, Xia gave a detailed proof for the analogous result for the Bergman space on the unit ball of $\mathbb C^n.$ Very recently, Difo \cite{D25}
 provided a detailed proof for the previous theorem.

We shall denote by $\mathcal{WL}$ (resp. $\mathcal{SL}, \rm{XZ-}\mathcal{SL})$ the class of weakly localized (resp. sufficiently localized (in the sense of Wang-Cao-Zhu), sufficiently localized in the sense of Xia-Zheng) operators on $H^2 (\mathbb C^n, d\mu).$ We shall show below the following two inclusions.
$$\mathcal{SL}\subset \rm{XZ-}\mathcal{SL}\subset \mathcal{WL}.$$
We denote by $C^\star (\mathcal{WL})$ (resp. $C^\star (\mathcal{SL}), C^\star (\rm{XZ-}\mathcal{SL}))$ the $C^\star-$algebra generated by the class $\mathcal{WL}$ (resp. $\mathcal{SL}, \rm{XZ-}\mathcal{SL}).$ It follows from Theorem \ref{xia} that
$$C^\star (\mathcal{SL})\subset C^\star (\rm{XZ-}\mathcal{SL})\subset C^\star (\mathcal{WL})=\mathcal T^1\subset C^\star (\mathcal{SL}).$$
This implies the equality $$C^\star (\mathcal{WL})=\mathcal T^1=C^\star (\mathcal{SL})=C^\star (\rm{XZ-}\mathcal{SL}).$$
This raises the following question: are the inclusions $\mathcal{SL}\subset \rm{XZ-}\mathcal{SL}\subset \mathcal{WL}$ strict?

We settle these questions in the affirmative. Our main result is the following.

\begin{thm}\label{thm13}
The following two inclusions are strict.
$$\mathcal{SL}\subsetneq \rm{XZ-}\mathcal{SL}\subsetneq \mathcal{WL}.$$ 
\end{thm}

We point out that all the examples provided in \cite{XZ13}
 of XZ-sufficiently localized operators are sufficiently localized (in the sense of Wang-Cao-Zhu). The proof of the first set inequality is 'direct': we exhibit a class of examples of XZ-sufficiently localized operators which are not sufficiently localized (in the sense of Wang-Cao-Zhu). The proof of the second set inequality is not constructive: it relies on an application of the uniform boundedness principle and the Riesz representation theorem for the dual of $L^1.$

The plan of the paper is as follows. In Section \ref{sec2}
, we give the definitions of the classes of localized operators; in that section, following Sadeghi-Zorboska \cite{SZ21}
  we also introduce the notion of strongly localized operator. It is clear that every strongly localized operator is sufficiently localized in the sense of Wang, Cao and Zhu. 
In Section \ref{sec3}
, we establish properties of localized operators. In particular, we prove that the class of sufficiently localized operators introduced by Wang, Cao and Zhu \cite{WCZ13}
 is contained in the class of sufficiently localized operators introduced by Xia and Zheng \cite{XZ13}
  and the latter class is contained in the class of weakly localized operators. In Section \ref{sec4}
  , we discuss different examples of sufficiently localized operators, among which some studied in \cite{WCZ13}
   and \cite{XZ13}.
    Wang, Cao and Zhu \cite[Section 3]{WCZ13}
     exhibited examples of composition operators which are sufficiently localized, but are not strongly localized. In this section, we shall revisit these examples (Theorem \ref{thm42}
     ). In Section \ref{sec5}
     , we investigate the localization properties of the class of Toeplitz operators with complex measure symbols whose total variation measure is a Fock-Carleson measure.
In Section \ref{sec6}
, we study an example of a bounded linear operator on $H^2 (\mathbb C^n, d\mu)$ which is not weakly localized although its Berezin transform vanishes at infinity; to this aim, we show that it is not compact and we conclude through a combination of Theorems \ref{thm11} 
and \ref{thm12}.
 This operator is not even a member of the Toeplitz algebra. We further study a class of operators of the same type which are compact on $H^2 (\mathbb C^n, d\mu).$ In Section \ref{sec7}
 , we study the class of composition operators on $H^2 (\mathbb C^n, d\mu).$

In Section \ref{sec8}
, we study the class of singular operators of convolution type introduced by Zhu \cite{Z15}
. Our main result (Theorem \ref{thm13}
) is proved in Section \ref{sec9} 
(Corollary \ref{cor95}
and Corollary \ref{cor98}
) in terms of these singular operators of convolution type. Finally, in Section \ref{sec10}
, we list the questions we were not able to solve in this paper.

Throughout this article, we use $C$ to denote a positive constant whose value may change from line to line, but does not depend on functions being considered and the constants on which it depends will be specified in parentheses when necessary. We shall also denote by $p^\prime$ the conjugate exponent of $p\in (1, \infty).$

\section{Definitions}\label{sec2}
\subsection{Admissible operators on $H^2 (\mathbb C^n, d\mu)$}\label{ssec21}
We start with the definition of an admissible operator on $H^2 (\mathbb C^n, d\mu).$ We refer to \cite{WCZ13} 
and \cite{MW14}.
.

\begin{defn}\label{def21}
We denote by $\mathcal S$ the linear span of the normalized reproducing kernels $k_z.$ A linear operator $T: \mathcal S\rightarrow H^2 (\mathbb C^n, d\mu)$ is said to be {\textit {admissible on}} $H^2 (\mathbb C^n, d\mu)$ if there exists a linear operator $T^\star: \mathcal S\rightarrow H^2 (\mathbb C^n, d\mu)$ such that the duality relation
\begin{equation}\label{eq21}
\langle Tk_z, k_w\rangle=\langle k_z, T^\star k_w\rangle
\end{equation}
holds for all $z, w\in \mathbb C^n.$
\end{defn}

In view of the atomic decomposition theorem on $H^2 (\mathbb C^n, d\mu)$ (\cite[Theorem 8.2]{JPZ17}
 for $n\geq1$, \cite[Theorem 2.34]{Z12}
 for $n=1$), the subspace $\mathcal S$ is dense in $H^2 (\mathbb C^n, d\mu).$ Moreover, every bounded linear operator on $H^2 (\mathbb C^n, d\mu)$ is admissible.

We next show the following proposition \cite[Lemma 2.3]{WCZ13}.

\begin{prop}\label{pro22}
Every admissible operator $T$ on  $H^2 (\mathbb C^n, d\mu)$ extends to an integral operator $\overbrace{T}$ on $L^2_1 (\mathbb C^n, dV)$ defined as follows. For all $z\in \mathbb C^n$ and $f\in L^2_1 (\mathbb C^n, dV),$ we have
$$\overbrace{T}f(z)=\int_{\mathbb C^n} R(z, w)f(w)d\mu (w),$$
where $R(z, w):=\overline{T^\star K_z (w)}.$
\end{prop}

\begin{proof}
For all $z, w\in \mathbb C^n,$ the relation (2.1) 
 is equivalent to 
\begin{equation}\label{eq22}
\langle Tk_w, K_z\rangle=\langle k_w, T^\star K_z\rangle.
\end{equation}
For $f\in \mathcal S,$ we have $f=\sum_{j=1}^m a_jk_{z_j},$ with $a_j\in \mathbb C,\hskip 1truemm z_j\in \mathbb C^n.$ Using the reproducing kernel property and the relation \eqref{eq22}, it holds

\begin{eqnarray*}
Tf(z)&= &\langle Tf, K_z\rangle\\
&=&\sum_{j=1}^m a_j\langle Tk_{z_j}, K_z\rangle=\sum_{j=1}^m a_j\langle k_{z_j}, T^\star K_z\rangle\\
&=&\langle \sum_{j=1}^m a_jk_{z_j}, T^\star K_z\rangle=\langle f, T^\star K_z\rangle\\
&=&\int_{\mathbb C^n} f(w)\overline{T^\star K_z (w)}d\mu (w).
\end{eqnarray*}

Since $T^\star K_z\in H^2 (\mathbb C^n, d\mu),$ it is easy to check that the operator $T$ on  $H^2 (\mathbb C^n, d\mu)$ extends to the integral operator $\overbrace{T}$ on $L^2_1 (\mathbb C^n, dV)$ defined as 
$$\overbrace{T}f(z)=\int_{\mathbb C^n} R(z, w)f(w)d\mu (w),\quad z\in \mathbb C^n,$$
where $R(z, w):=\overline{T^\star K_z (w)}.$
\end{proof}

As usual, we shall denote again by $T$ the extension $\overbrace{T}$ of the admissible operator $T.$

\subsection{Sufficiently, strongly and weakly localized operators on $H^2 (\mathbb C^n, d\mu)$}\label{ssec22}
The notion of XZ-{\textit {sufficiently localized operators}} was introduced in \cite{XZ13}.

\begin{defn}\label{def23}
An admissible operator $T$ on $H^2 (\mathbb C^n, d\mu)$ is said to be XZ-{\it {sufficiently localized}} if there exist constants $2n<\beta<\infty$ and $0<C<\infty$ such that
$$\left \vert \langle Tk_z, k_w\rangle\right \vert \leq \frac C{\left (1+\left \vert z-w\right \vert \right )^\beta}$$
for all $z, w\in \mathbb C^n.$
\end{defn}

We denote by XZ-$\mathcal{SL}$ the collection  of XZ-sufficiently localized  operators on $H^2 (\mathbb C^n, d\mu).$

We next recall the notion of {\textit {sufficiently localized operator}} introduced in \cite{XZ13}.
We first recall that, for every $z\in \mathbb C^n,$ the unitary operator $U_z$ is defined on $H^2 (\mathbb C^n, d\mu)$ by
$$U_z f(w)=f(z-w)k_z (w), \qquad f\in H^2 (\mathbb C^n, d\mu).$$
\begin{defn}\label{def24}
Let  $p$ be a positive real.
\begin{enumerate}
\item
An admissible operator $T$ on $H^2 (\mathbb C^n, d\mu)$ is said to be $p$-{\textit {localized}} if 
$$\sup \limits_{z\in \mathbb C^n} \int_{\mathbb C^n} \left \vert U_zTU_z \mathbbm{1} (w)\right \vert^pd\mu(z)<\infty.$$
Here, the symbol $\mathbbm{1}$ denotes the constant function on $\mathbb C^n,$ identically equal to $1.$ 
\item
An admissible operator $T$ on $H^2 (\mathbb C^n, d\mu)$ is said to be {\textit {sufficiently localized}} if it is $p$-localized for some $p> 2.$ It is said to be {\textit {strongly localized}} if it is $p$-localized for all $p> 2.$
%
\end{enumerate}
\end{defn}

We call $\mathcal L_p$ the space of $p$-localized operators,  $\mathcal{SL}$ the space of sufficiently localized operators and $\mathcal {S}t\mathcal {SL}$ the space of strongly localized operators. Clearly, 
$$\mathcal {SL}=\cup_{2< p<\infty} \mathcal L_p, \hskip 2truemm \mathcal {S}t\mathcal {SL}=\cap_{2< p<\infty} \mathcal L_p\,\, \textrm{and}\,\, \mathcal {S}t\mathcal {SL}\subset \mathcal {SL}.$$

The notion of {\textit {weakly localized operator}} was introduced in \cite{IMW15}.

\begin{defn}\label{def25}
An admissible operator $T$ on $H^2 (\mathbb C^n, d\mu)$ is said to be {\textit {weakly localized}} if it satisfies the conditions 
$$\sup \limits_{z\in \mathbb C^n} \int_{\mathbb C^n} \left \vert \langle Tk_z, k_w\rangle\right \vert dV(w)<\infty, \qquad \sup \limits_{z\in \mathbb C^n} \int_{\mathbb C^n} \left \vert \langle T^\star k_z, k_w\rangle \right \vert dV(w)<\infty$$
and
$$\lim \limits_{r\rightarrow \infty} \sup \limits_{z\in \mathbb C^n} \int_{|z-w|\geq r} \left \vert \langle Tk_z, k_w\rangle\right \vert dV(w)=0, \qquad \lim \limits_{r\rightarrow \infty} \sup \limits_{z\in \mathbb C^n} \int_{|z-w|\geq r} \left \vert \langle T^\star k_z, k_w\rangle \right \vert dV(w)=0.$$
\end{defn}

\section{Properties of localized operators}\label{sec3}
\subsection{Preliminary properties}\label{ssec31}
In view of the H\"older inequality, 
we get the following proposition.

\begin{prop}\label{pro31}
For all $0<q<p<\infty,$ the following inclusion holds: 
$$\mathcal L_p \subset \mathcal L_q.$$
\end{prop}

It is easy to check that for every $p>2,$ the collection $\mathcal L_p$ is a vector subspace of the complex vector space of all linear operators on $H^2 \left (\mathbb C^n, d\mu\right ).$ We consider the following functional on $\mathcal L_p:$
$$\left \Vert T\right \Vert_{\mathcal L_p}=\sup \limits_{z\in \mathbb C^n} \left \Vert U_zTU_z \mathbbm 1\right \Vert_{L^p (\mathbb C^n, d\mu)}.$$

We shall also need the following lemma.

\begin{lem}\label{lem32}
Let $T$ be an admissible operator on $H^2 (\mathbb C^n, d\mu).$ Let $z\in \mathbb C^n.$  For all $p>0,$ we have
\begin{equation}\label{eq31}
\int_{\mathbb C^n} \left \vert U_zTU_z \mathbbm 1(w)\right \vert^pd\mu (w)=\frac 1{\pi^n}\int_{\mathbb C^n} \left \vert \langle Tk_z, k_w\rangle\right \vert^pe^{\left (\frac p2-1\right )|z-w|^2}2dV(w).
\end{equation}
\end{lem}

\begin{proof}
Using the change of variable $t=z-w,$ the equality $K_t=k_te^{\frac {|t|^2}2}$ and the reproducing kernel property, it holds
\begin{eqnarray*}
\int_{\mathbb C^n} \left \vert U_zTU_z \mathbbm 1(w)\right \vert^pd\mu (w)
&= &\int_{\mathbb C^n} \left \vert U_zTk_z (w)\right \vert^pd\mu (w)\\
&=&\frac 1{\pi^n}\int_{\mathbb C^n} \left \vert Tk_z (z-w)k_z (w)\right \vert^pe^{-|w|^2}dV (w)\\
&=&\frac 1{\pi^n}\int_{\mathbb C^n} \left \vert Tk_z (t)k_z (z-t)\right \vert^pe^{-|z-t|^2}dV (t)\\
&=&\frac 1{\pi^n}\int_{\mathbb C^n} \left \vert \langle Tk_z, K_t)\rangle \right \vert^p \left \vert k_z (z-t)\right \vert^pe^{-|z-t|^2}dV (t)\\
&=&\frac 1{\pi^n}\int_{\mathbb C^n} \left \vert \langle Tk_z, k_t)\rangle \right \vert^p e^{\frac {p|t|^2}2}\left \vert e^{\langle z-t, z\rangle-\frac {|z|^2}2}\right \vert^pe^{-|z-t|^2}dV (t)\\
&=&\frac 1{\pi^n}\int_{\mathbb C^n} \left \vert \langle Tk_z, k_t)\rangle \right \vert^p e^{\frac {p|z-t|^2}2}e^{-|z-t|^2}dV (t).
\end{eqnarray*}

\end{proof}

We next prove the following proposition.

\begin{prop}\label{pro33}
For each $p\in (2, \infty),$ the pair $\left (\mathcal L_p, \left \Vert \cdot\right \Vert_{\mathcal L_p} \right )$ is a normed complex vector space.
\end{prop} 

\begin{proof}
The proof of the identity $\left \Vert \lambda T\right \Vert_{\mathcal L_p}=|\lambda|\left \Vert T\right \Vert_{\mathcal L_p}$ and the triangular inequality are left as an exercise. We just provide the proof of  the implication $\left \Vert T\right \Vert_{\mathcal L_p}=0 \Rightarrow T=0.$ In view of Lemma \ref{lem32}, the assumption $\left \Vert T\right \Vert_{\mathcal L_p}=0$ is equivalent to the following statement: for all $w, z\in \mathbb C^n,$ we have
$$0=\langle Tk_w, k_z\rangle=\langle k_w, T^\star k_z\rangle=e^{-\frac {|w|^2+|z|^2}2}\overline{T^\star K_z (w)},$$
or equivalently, $\overline{T^\star K_z (w)}=R(z, w)\equiv 0.$ The kernel $R$ was defined in Proposition \ref{pro22}. 
It follows from the same proposition that $T=0.$
\end{proof}

\subsection{Properties of weakly localized operators}\label{ssec32}
We now prove the following theorem. For $q=2$ and for sufficiently localized operators, its second part was proved  in \cite[Theorem  2.4]{WCZ13}
; it was also proved in \cite[Theorem  3.1]{MW14}
 for more general Bergman-type spaces. 

\begin{thm}\label{thm34}
Every admissible operator on $H^2 (\mathbb C^n, d\mu)$ which satisfies the first two conditions of weakly localized operators  extends into a bounded linear operator on the Fock spaces $F^q_1$ for every $q\in [1, \infty].$ It also extends into a bounded linear operator from $L^2_1 (\mathbb C^n, dV)$ to $F^2_1.$ 
\end{thm}

\begin{proof}
Let $T$ be an admissible operator on $H^2 (\mathbb C^n, d\mu)$ which satisfies the first two conditions of weakly localized operators. We write:
$$C:=\max \left \{\sup \limits_{z\in \mathbb C^n} \int_{\mathbb C^n} \left \vert \langle Tk_z, k_w\rangle\right \vert dV(w), \quad \sup \limits_{z\in \mathbb C^n} \int_{\mathbb C^n} \left \vert \langle T^\star k_z, k_w\rangle \right \vert dV(w)\right \}<\infty.$$

We first prove that $T$ extends into a bounded linear operator on $F^1_1.$ Since the subspace $\mathcal S$ defined in subsection \ref{ssec21} 
 is also dense in $F^1_1,$ it suffices to show that there exists a positive constant $c$ such that for all $f\in \mathcal S,$ we have
\begin{equation}\label{eq32}
\left \Vert Tf\right \Vert_{F^1_1}\leq c\left \Vert f\right \Vert_{F^1_1}.
\end{equation}
According to Proposition \ref{pro22}, we have 
\begin{equation*}
Tf(z)=\int_{\mathbb C^n} f(w)\overline{T^\star K_z (w)}d\mu (w)
\end{equation*} 
and then
\begin{eqnarray}
\left \vert Tf(z)e^{-\frac {|z|^2}2}\right \vert&=&\left \vert \int_{\mathbb C^n} f(w)\overline{T^\star k_z (w)}d\mu (w)\right \vert \nonumber\\
&\leq &\int_{\mathbb C^n} |f(w)|\left \vert T^\star k_z (w)\right \vert d\mu (w) \nonumber\\
&=& \int_{\mathbb C^n} |f(w)|\left \vert \langle T^\star k_z, K_w \rangle\right \vert d\mu (w) \nonumber\\
&=&\frac 1{\pi^n}\int_{\mathbb C^n} |f(w)|\left \vert \langle T^\star k_z, k_w \rangle\right \vert e^{-\frac {|w|^2}2}dV(w).\label{eq33}
\end{eqnarray}
By the Fubini-Tonelli theorem, it follows that
\begin{eqnarray*}
\int_{\mathbb C^n} \left \vert Tf(z)e^{-\frac {|z|^2}2}\right \vert dV(z)
&\leq & \frac 1{\pi^n}\int_{\mathbb C^n} \left \{\int_{\mathbb C^n} |f(w)|\left \vert \langle T^\star k_z, k_w \rangle\right \vert e^{-\frac {|w|^2}2}dV(w)\right \}dV(z)\\
&=&\frac 1{\pi^n}\int_{\mathbb C^n} |f(w)|\left (\int_{\mathbb C^n} \left \vert \langle Tk_w, k_z \rangle\right \vert dV(z)\right )e^{-\frac {|w|^2}2}dV(w)\\
&\leq& \frac 1{\pi^n}\int_{\mathbb C^n} |f(w)|\left (\sup \limits_{w\in \mathbb C^n}\int_{\mathbb C^n} \left \vert \langle Tk_w, k_z \rangle\right \vert dV(z)\right )e^{-\frac {|w|^2}2}dV(w)\\
&\leq&\frac C{\pi^n}\int_{\mathbb C^n} |f(w)|e^{-\frac {|w|^2}2}dV(w).
\end{eqnarray*}

This finishes the proof of \eqref{eq32}.

Furthermore, for all $f\in F^\infty_1,$ we have from \eqref{eq32} that
$$\sup \limits_{z\in \mathbb C^n} \left \vert Tf(z)e^{-\frac {|z|^2}2}\right \vert\leq \frac {\left \Vert f\right \Vert_{F^1_1}}{\pi^n} \sup \limits_{z\in \mathbb C^n} \int_{\mathbb C^n} \left \vert \langle T^\star k_z, k_w \rangle\right \vert dV(w)\leq \frac C{\pi^n}\left \Vert f\right \Vert_{F^1_1}<\infty.$$
Hence, $T$ is bounded on $F^1_1$ and on $F^\infty_1.$ It follows from the complex interpolation (\cite[Theorem 7.3]{JPZ17} for $n\geq1,$ \cite[Theorem 2.30]{Z12} for $n=1)$ that $T$ is bounded on $F^q_1$ for all $1\leq q \leq \infty.$ In particular, $T$ is bounded on $F^2_1=H^2 (\mathbb C^n, d\mu).$

We prove next that $T$ extends into a bounded linear operator from $L^2_1 (\mathbb C^n, dV)$ to $F^2_1.$ According to Proposition \ref{pro22}, it suffices to show that the 'positive' integral operator $T^+$ defined as
$$T^+f(z)=\int_{\mathbb C^n} |T^\star K_z (w)|f(w)d\mu (w)$$
is bounded on $L^2_1 (\mathbb C^n, dV).$ This follows through Schur's lemma. Indeed, with the Schur test function $h(w)=e^{\frac {|w|^2}2},$ on the one hand, we have
\begin{eqnarray*}
\int_{\mathbb C^n} |T^\star K_z (w)|e^{\frac {|w|^2}2}d\mu (w)
&=&\frac 1{\pi^n}\int_{\mathbb C^n} |\langle T^\star k_z, k_w\rangle|e^{\frac {|z|^2}2}e^{\frac {|w|^2}2}e^{\frac {|w|^2}2}e^{-|w|^2}dV(w)\\
&=&\frac 1{\pi^n}e^{\frac {|z|^2}2}\int_{\mathbb C^n} |\langle T^\star k_z, k_w\rangle|dV(w)\\
&\leq& \frac C{\pi^n}e^{\frac {|z|^2}2},
\end{eqnarray*}

where $C=\sup \limits_{z\in \mathbb C^n}\int_{\mathbb C^n} \left \vert \langle T^\star k_w, k_z \rangle\right \vert dV(w)<\infty.$ On the second hand, we have

\begin{eqnarray*}
\int_{\mathbb C^n} |T^\star K_z (w)|e^{\frac {|z|^2}2}d\mu (z)
&=&\frac 1{\pi^n}\int_{\mathbb C^n} |\langle T^\star k_z, k_w\rangle|e^{\frac {|z|^2}2}e^{\frac {|w|^2}2}e^{\frac {|z|^2}2}e^{-|z|^2}dV(z)\\
&=&\frac 1{\pi^n}e^{\frac {|w|^2}2}\int_{\mathbb C^n} |\langle T^\star k_z, k_w\rangle|dV(z)\\
&=&\frac 1{\pi^n}e^{\frac {|w|^2}2}\int_{\mathbb C^n} |\langle Tk_w, k_z\rangle|dV(z)\\
&\leq& \frac C{\pi^n}e^{\frac {|w|^2}2}.
\end{eqnarray*}

This finishes the proof of the theorem.
\end{proof}

We next prove the following lemma.

\begin{lem}\label{lem35}
Every admissible operator on $H^2 (\mathbb C^n, d\mu)$ which fulfils the first two conditions of weakly localized operators is a member of $\mathcal L_2.$
\end{lem}

\begin{proof}
Let $T$ be an admissible operator on $H^2 (\mathbb C^n, d\mu)$ which fulfils the first two conditions of weakly localized operators. We recall that each $U_z$ is an isometry on $H^2 (\mathbb C^n, d\mu).$ From the previous theorem, we know that $T$ is bounded on $H^2 (\mathbb C^n, d\mu).$ Therefore, for all $z\in \mathbb C^n,$ 
\begin{eqnarray*}
\int_{\mathbb C^n} \left \vert U_zTU_z \mathbbm 1(w)\right \vert^2 d\mu (w)
&=&\left \Vert U_zTU_z\right \Vert^2_{H^2 (\mathbb C^n, d\mu)}
\leq \left \Vert U_z\right \Vert^2 \left \Vert T\right \Vert^2 \left \Vert U_z\right \Vert^2 \int_{\mathbb C^n} \left \vert \mathbbm 1(w)\right \vert^2 d\mu (w)\\
&=&\left \Vert T\right \Vert^2.
\end{eqnarray*}
Whence
$$\sup \limits_{z\in \mathbb C^n} \int_{\mathbb C^n} \left \vert U_zTU_z \mathbbm 1(w)\right \vert^2 d\mu (z)\leq \left \Vert T\right \Vert^2<\infty.$$
This completes the proof.
\end{proof} 

We end this subsection with the following proposition.

\begin{prop}\label{pro36}
Let $T\in \mathcal L (H^2 (\mathbb C^n, d\mu)).$ Suppose that $T$ satisfies the last two conditions of weakly localized operators. Then $T$ satisfies the first two conditions of weakly localized operators.
\end{prop}

\begin{proof}
By definition, $\lim \limits_{r\rightarrow \infty} \sup \limits_{z\in \mathbb C^n} \int_{|z-w|\geq r} \left \vert \langle Tk_z, k_w\rangle\right \vert dV(w)=0$ means that: for every $\epsilon>0,$ there exists $\delta >0$ such that for every positive number $r$, the following implication holds
$$r\geq \delta \Rightarrow \sup \limits_{z\in \mathbb C^n} \int_{|z-w|\geq r} \left \vert \langle Tk_z, k_w\rangle\right \vert dV(w)<\epsilon.$$
In particular, for $\epsilon =1,$ there exists $\delta_1>0$ such that
\begin{equation}\label{eq34}
\sup \limits_{z\in \mathbb C^n} \int_{|z-w|\geq \delta_1} \left \vert \langle Tk_z, k_w\rangle\right \vert dV(w)<1.
\end{equation}
Since $T$ is bounded on $H^2 (\mathbb C^n, d\mu),$ then using the Cauchy-Schwarz inequality, we have
\begin{eqnarray}
\int_{|z-w|<\delta_1} \left \vert \langle Tk_z, k_w\rangle\right \vert dV(w)&\leq& \int_{|z-w|<\delta_1}  \left \Vert Tk_z\right \Vert \left \Vert k_w\right \Vert dV(w) \nonumber\\
&\leq& \left \Vert T\right \Vert \int_{|z-w|<\delta_1} dV(w) \nonumber\\
&=&c_n\delta_1^{2n}\left \Vert T\right \Vert. \label{eq35}
\end{eqnarray}
We deduce from \eqref{eq34} and \eqref{eq35} that
$$\sup \limits_{z\in \mathbb C^n} \int_{\mathbb C^n} \left \vert \langle Tk_z, k_w\rangle\right \vert dV(w)<1+c_n\delta_1^{2n}<\infty.$$
Since $T^\star$ is also bounded as well as $T,$ we conclude that $T$ satisties the first two conditions of weakly localized operators.
\end{proof} 

\subsection{Sufficiently localized operators are XZ-sufficiently localized}\label{ssec33}
In this subsection, we prove that sufficiently localized operators  are XZ-sufficiently localized.  This fact was established in \cite[Lemma 2.2]{WCZ13}
. We give the proof for completeness. 
For $\alpha, p >0,$  let $L^p_\alpha (\mathbb C^n, dV)$ be the Lebesgue space of measurable functions $f$ on $\mathbb C^n$ such that
$$\left \Vert f\right \Vert_{L_\alpha^p}^p=\left (\frac p{2\pi} \right )^n\int_{\mathbb C^n} \left \vert f(z)e^{-\frac {\alpha|z|^2}2}\right \vert^pdV(z)<\infty.$$
The Fock space $F^p_\alpha$ is the space of entire functions on $\mathbb C^n$ which belong to $L^p_\alpha (\mathbb C^n, dV).$ 

We next record the following lemma (\cite[Lemma 2.1]{IZ10}
  and \cite[Lemma 2.32]{Z12} 
  for $n=1.$) 

\begin{lem}\label{lem37}
Let $\alpha, r$ and $p$ be  positive numbers. Then there exists some positive constant $C=C(\alpha, r, p)$ such that
$$\left \vert f(z)e^{-\frac {\alpha |z|^2}2}\right \vert^p\leq C\int_{B(z, r)} \left \vert f(w)e^{-\frac {\alpha |w|^2}2}\right \vert^p dV(w)$$
for all entire functions in $\mathbb C^n.$
\end{lem}

We rely on the following corollary (\cite[Corollary 2.8]{Z12} for $n=1).$ 

\begin{cor}\label{cor38}
Let $\alpha, p >0$ and $f\in F^p_\alpha.$ Then for all $z\in \mathbb C^n,$ we have
$$|f(z)|\leq \left \Vert f\right \Vert_{F^p_\alpha}e^{\frac {\alpha|z|^2}2}.$$
\end{cor}

We shall also use the following elementary lemma.

\begin{lem}\label{lem39}
Let $\beta$ and $\epsilon$ be two positive numbers such that $\beta>\epsilon.$ Then for all $x>0,$ it holds:
$$e^{-\epsilon x}\leq \frac {\left (\frac \beta \epsilon\right )^\beta}{ (1+x)^\beta}.$$
\end{lem}

We prove the following proposition.

\begin{prop}\label{pro310}
Sufficiently localized operators are {\rm {XZ}}-sufficiently localized.
\end{prop}

\begin{proof}
Let $T$ be a sufficiently localized operator. Then for all $z\in \mathbb C^n,$ we have $U_zTU_z \mathbbm 1\in F^p_\alpha$ for some $p>2,$ with $\alpha=\frac 2p.$ It follows from the previous corollary that for all $z, w\in \mathbb C^n,$ we have
$$\left \vert U_zTU_z \mathbbm 1(w)\right \vert \leq \left \Vert U_zTU_z \mathbbm 1\right \Vert_{F^p_{\frac 2p}}e^{\frac {|w|^2}p}.$$
We set $C=\sup \limits_{z\in \mathbb C^n} \left \Vert U_zTU_z \mathbbm 1\right \Vert_{F^p_{\frac 2p}}<\infty.$ We shall use the following identity:
$$U_zTU_z \mathbbm 1(w)=k_z (w)(Tk_z)(z-w),$$
from which we deduce that
$$\left \vert k_z (w)(Tk_z)(z-w)\right \vert \leq Ce^{\frac {|w|^2}p}.$$
We apply the change of variable $z-w=\zeta;$ the latter inequality becomes
$$\left \vert k_z (z-\zeta)(Tk_z)(\zeta)\right \vert \leq Ce^{\frac {|z-\zeta|^2}p}.$$
This amounts to
$$\left \vert e^{\langle z-\zeta, z\rangle}e^{-\frac {|z|^2}2}\langle Tk_z, k_\zeta\rangle e^{\frac {|\zeta|^2}2}\right \vert \leq Ce^{\frac {|z-\zeta|^2}p}.$$
We easily reach to the following inequality:
$$\left \vert \langle Tk_z, k_\zeta\rangle \right \vert \leq Ce^{\left (\frac 1p-\frac 12 \right )|z-\zeta|^2}.$$
The required conclusion follows from an application of Lemma  \ref{lem39}
since $\frac 1p-\frac 12<0.$ 
\end{proof}

From the proof of the previous proposition, we deduce a necessary and sufficient condition for an admissible operator on $H^2 (\mathbb C^n, d\mu)$ to be sufficiently localized.

\begin{cor}\label{cor311}
Let $T$ be admissible operator on $H^2 (\mathbb C^n, d\mu).$ The following three conditions are equivalent.
\begin{enumerate}
\item
$T$ is sufficiently localized;
\item
There exist some $\epsilon >0$ and a positive constant $C=C(\epsilon)$ such that for all $z, w\in \mathbb C^n,$ the following estimate holds:
$$\left \vert \langle Tk_z, k_w\rangle\right \vert \leq Ce^{-\epsilon |z-w|^2}.$$
\item
$T^\star$ is sufficiently localized.
\end{enumerate}
\end{cor}

\begin{proof}
$(1)\Rightarrow (2)$ For $T\in \mathcal L_p, p>2,$ we showed in the proof of the previous proposition that $T$ satisfies the implication (2) with $\epsilon=\frac 12-\frac 1p.$ 

$(2)\Rightarrow (1)$ According to the equation \eqref{eq31}, we have
$$
\int_{\mathbb C^n} \left \vert U_zTU_z \mathbbm 1(w)\right \vert^pd\mu (w)=\frac 1{\pi^n}\int_{\mathbb C^n} \left \vert \langle Tk_z, k_w\rangle\right \vert^pe^{\frac {p|z-w|^2}2}e^{-|z-w|^2}dV(w).
$$
Without loss of generality, we may suppose that $\epsilon <\frac 12.$ 
The assertion (1) follows easily with $2<p<\frac 1{\frac 12-\epsilon}.$ 

\item
$(2)\Leftrightarrow (3)$
We notice that $\left \vert\langle T^\star k_z, k_w\rangle \right \vert=\left \vert\langle Tk_w, k_z\rangle \right \vert.$ So the equivalence $(1)\Leftrightarrow (2)$ is valid with $T^\star$ in the place of $T.$
\end{proof}

We also deduce the following corollary.

\begin{cor}\label{cor312}
$\mathcal{SL}$ is a $\ast$-algebra.
\end{cor}

\begin{proof}
Since $T^\star \in \mathcal{SL}$ if $T\in \mathcal{SL}$ according to Corollary \ref{cor311}
, it remains to show that $T_2T_1\in \mathcal{SL}$ whenever $T_1, T_2\in \mathcal{SL}.$ According to Corollary \ref{cor311}
, we suppose that there exist some $\epsilon >0$ and a positive constant $C=C(\epsilon)$ such that for all $z, w\in \mathbb C^n,$ the following estimate holds:
\begin{equation}\label{eq36}
\left \vert \langle T_jk_z, k_w\rangle\right \vert \leq Ce^{-\epsilon |z-w|^2}, \quad j=1, 2.
\end{equation}
We must prove that there exists some $\eta >0$ and a positive constant $C'=C'(\eta)$ such that for all $z, w\in \mathbb C^n,$ the following estimate holds:
$$\left \vert \langle T_2T_1 k_z, k_w\rangle\right \vert \leq C'e^{-\eta |z-w|^2}.$$
We have
\begin{eqnarray*}
\langle T_2T_1 k_z, k_w\rangle
&=&\langle T_1 k_z, T_2^\star k_w\rangle\\
&=&\frac 1{\pi^n}\int_{\mathbb C^n} T_1 k_z (\zeta)\overline{T_2^\star k_w (\zeta)}e^{-|\zeta|^2}dV(\zeta)\\
&=&\int_{\mathbb C^n} \langle T_1 k_z, k_\zeta\rangle \overline{\langle T_2^\star k_w, k_\zeta\rangle}dV(\zeta)\\
&=&\int_{\mathbb C^n} \langle T_1 k_z, k_\zeta\rangle \overline{\langle k_w, T_2 k_\zeta\rangle}dV(\zeta).
\end{eqnarray*}
Then
\begin{eqnarray*}
\left \vert \langle T_2T_1 k_z, k_w\rangle \right \vert
&\leq& \int_{\mathbb C^n} \left \vert\langle T_1 k_z, k_\zeta\rangle \right \vert \left \vert \langle T_2 k_\zeta, k_w\rangle \right \vert dV(\zeta)\\
&\leq& C^2\int_{\mathbb C^n} e^{-\epsilon \left (|z-\zeta|^2+|w-\zeta|^2\right )}dV(\zeta).
\end{eqnarray*}
For the latest inequality, we applied \eqref{eq36}. Since $|z-\zeta|^2+|w-\zeta|^2\geq \frac 12|z-w|^2,$ we obtain:

\begin{eqnarray*}
\left \vert \langle T_2T_1 k_z, k_w\rangle \right \vert
&\leq& C^2 e^{-\frac \epsilon 4|z-w|^2}\int_{\mathbb C^n} e^{-\frac \epsilon 2|z-\zeta|^2}dV(\zeta)\\
&=&c_\epsilon C^2 e^{-\frac \epsilon 4|z-w|^2},
\end{eqnarray*}

where $c_\epsilon:=\int_{\mathbb C^n} e^{-\frac \epsilon 2|z-\zeta|^2}dV(\zeta)<\infty.$ The conclusion follows with $\eta=\frac \epsilon 4$ and $C'=c_\epsilon C^2.$
\end{proof}

We deduce the following corollary.

\begin{cor}\label{cor313}
The $C^\star-$algebra $C^\star (\mathcal{SL})$ generated in $\mathcal L(H^2 (\mathbb C^n, d\mu))$ by the $\ast$-algebra $\mathcal{SL}$ of sufficiently localized operators is equal to the operator-norm closure of $\mathcal{SL}.$
\end{cor}

We end this subsection with the following remark.

\begin{rem}\label{rem314}
All the examples of XZ-sufficiently localized operators given in \cite{XZ13} 
belong to $\mathcal{SL}.$ This raises the following question: is the inclusion $\mathcal{SL}\subset XZ-\mathcal{SL}$ strict?
\end{rem}

\subsection{XZ-sufficiently localized operators are weakly localized}\label{ssec34}
In this subsection, we prove the following proposition.

\begin{prop}\label{pro315}
Every {\rm {XZ}}-sufficiently localized operator is weakly localized.
\end{prop}

\begin{proof}
Let $T$ be a XZ-sufficiently localized operator. In view of Lemma \ref{lem35}
, it suffices to show that $T$ satisfies the last two conditions of a weakly localized operator. 
Let $z\in \mathbb C^n;$ using the change of variables $\zeta=z-w$ and the spherical coordinates, it holds

\begin{eqnarray*}
\int_{|z-w|\geq r} \left \vert \langle Tk_z, k_w\rangle\right \vert dV(w)
&\leq& C\int_{|z-w|\geq r} \frac 1{(1+|z-w|)^\beta}dV(w)\\
&=&C\int_{|\zeta|\geq r} \frac 1{(1+|\zeta|)^\beta}dV(\zeta)\\
&=&Cc_n \int_r^\infty \frac {\rho^{2n-1}}{(1+\rho)^\beta}d\rho\\
&\leq& Cc_n \int_r^\infty \frac {1}{\rho^{\beta-2n+1}}d\rho\\
&=&\frac {Cc_n}{\beta-2n}\frac 1{r^{\beta-2n}}.
\end{eqnarray*}

It follows that
$$\lim \limits_{r\rightarrow \infty} \sup \limits_{z\in \mathbb C^n} \int_{|z-w|\geq r} \left \vert \langle Tk_z, k_w\rangle\right \vert dV(w)\leq \frac {Cc_n}{\beta-2n}\lim \limits_{r\rightarrow \infty} \frac 1{r^{\beta-2n}}=0.$$
Using the duality relation \eqref{eq21} of an admissible operator and the similar way with $T,$ we also have
\begin{eqnarray*}
\lim \limits_{r\rightarrow \infty} \sup \limits_{z\in \mathbb C^n} \int_{|z-w|\geq r} \left \vert \langle T^\star k_z, k_w\rangle\right \vert dV(w)
&=&\lim \limits_{r\rightarrow \infty} \sup \limits_{z\in \mathbb C^n} \int_{|z-w|\geq r} \left \vert \langle T k_w, k_z\rangle\right \vert dV(w)\\
&\leq& \frac {Cc_n}{\beta-2n}\lim \limits_{r\rightarrow \infty} \frac 1{r^{\beta-2n}}=0.
\end{eqnarray*}

\end{proof}

Combining the previous proposition with Proposition \ref{pro310}, we obtain at once the following corollary. 

\begin{cor}\label{cor316}
Every sufficiently localized operator is weakly localized.
\end{cor}

\subsection{The Toeplitz algebra $\mathcal T^1$ is a subspace of $\mathcal L_2$}\label{ssec35}
We prove the following theorem.

\begin{thm}\label{thm317} The following inclusions hold: $\mathcal T^1 \subset \mathcal L (H^2 (\mathcal C^n, d\mu))$ and $\mathcal T^1 \subset \mathcal L_2.$
\end{thm}

\begin{proof}
The first inclusion is well-known; it is even known to be strict. 

Let us prove the second inclusion. Let $T$ be a member of $\mathcal T^1.$ There is a sequence $\{f_j\}_{j=1}^\infty$ of bounded measurable functions in $\mathbb C^n$ such that $\{T_{f_j}\}_{j=1}^\infty \rightarrow T$ in the operator norm. 
We have
$$\langle Tk_z, k_w\rangle=\langle \left (\lim \limits_{j\rightarrow \infty} T_{f_j}\right ) k_z, k_w\rangle=\lim \limits_{j\rightarrow \infty} \langle T_{f_j}k_z, k_w\rangle =e^{-\frac {|w|^2}2} \lim \limits_{j\rightarrow \infty} \left (T_{f_j} k_z\right )(w).$$
So 
\begin{eqnarray*}
\int_{\mathbb C^n} \left \vert \langle Tk_z, k_w\rangle \right \vert^2dV(w)
& = &\int_{\mathbb C^n} \lim \limits_{j\rightarrow \infty} \left\vert \left(T_{f_j} k_z\right)(w)\right\vert^2d\mu (w)\\
&\leq & \liminf \limits_{j\rightarrow \infty} \int_{\mathbb C^n} \left \vert \left(T_{f_j} k_z\right)(w)\right\vert^2d\mu (w) \\
&\leq & \sup \limits_{j} \int_{\mathbb C^n} \left\vert \left (T_{f_j} k_z\right )(w)\right\vert^2d\mu (w) \\
&\leq& \sup \limits_{j} \left \Vert T_{f_j}\right \Vert^2 \left \Vert k_z \right \Vert^2_{H^2 (\mathbb C^n, d\mu)}=\sup \limits_{j} \left \Vert T_{f_j}\right \Vert^2.
\end{eqnarray*}

For the first inequality, we applied Fatou's Lemma. Next, from a consequence of the uniform boundedness principle (cf. e.g. \cite[Corollary 2.3]{B11}
), we obtain $\sup \limits_{j} \left \Vert T_{f_j}\right \Vert<\infty.$ Indeed, $\{T_{f_j}\}_{j=1}^\infty$ is a sequence of continuous linear operators on $H^2 (\mathbb C^n, d\mu)$ such that for every $F\in H^2 (\mathbb C^n, d\mu),$ $T_{f_j} F$ converges in $H^2 (\mathbb C^n, d\mu)$ (as $j\rightarrow \infty)$ to the limit $TF.$ We conclude that
$$\sup \limits_{z\in \mathbb C^n} \int_{\mathbb C^n} \left \vert \langle Tk_z, k_w\rangle \right \vert^2dV(w)\leq \sup \limits_{j} \left \Vert T_{f_j}\right \Vert^2<\infty.$$
The same result holds for $T^\star$ in the place of $T,$ since $(T_{f_j})^\star=T_{\overline{f_j}}$ and then $\{T_{\overline {f_j}}\}_{j=1}^\infty \rightarrow T^\star$ in the operator norm.
\end{proof}

We shall show later (Corollary \ref{cor89}
) that the second inclusion in the previous theorem is strict. To recap, we state the following corollary.

\begin{cor}\label{cor318}
The following inclusion hold:
\begin{eqnarray*}
\mathcal St\mathcal L \subsetneq \mathcal S\mathcal L \subset {\rm {XZ-}}\mathcal S\mathcal L\subset \mathcal W\mathcal L&\subsetneq \mathcal T^1&\subsetneq \mathcal L_2\\
& 
	\subsetneq
&\\
&\mathcal L(H^2 (\mathbb C^n, d\mu))& .
\end{eqnarray*}
\end{cor}

\section{A family of strongly localized operators. A class of sufficiently localized operators which are not strongly localized}\label{sec4}
\subsection{A family of strongly localized operators}\label{ssec41}
In this subsection, we prove that the family $\{V_a \}_{a\in \mathbb C^n}$ of the unitary operators defined on $H^2 (\mathbb C^n, d\mu)$ as follows
$$V_a f(\zeta)=f(\zeta-a)k_a (\zeta), \quad \zeta \in \mathbb C^n, \hskip 2truemm f\in H^2 (\mathbb C^n, d\mu)$$
 is a family of strongly localized operators. It was shown in \cite[Section 4]{XZ13}
  that these operators are XZ-sufficiently localized. In particular, for $a=0,$ the operator $V_0=I$ (the identity operator) is strongly localized.

\begin{prop}\label{pro41}
For each $a\in \mathbb C^n,$ the associated operator $V_a$ is strongly localized.
\end{prop}

\begin{proof}
Let $a\in \mathbb C^n.$ For all $z, \zeta \in \mathbb C^n,$ we have
\begin{eqnarray*}
V_a k_z (\zeta)&=&k_z (\zeta-a)k_a (\zeta)\\
&=&e^{\langle \zeta-a, z\rangle}e^{-\frac {|z|^2}2}e^{\langle \zeta, a\rangle}e^{-\frac {|a|^2}2}\\
&=&e^{\langle \zeta, z+a\rangle-\langle a, z\rangle-\frac 12|z+a|^2+\Re e\hskip 1truemm \langle a, z\rangle}\\
&=&e^{\langle \zeta, z+a\rangle-\frac 12|z+a|^2-i\Im m\hskip 1truemm \langle a, z\rangle}\\
&=&k_{z+a} (\zeta)e^{-i\Im m\hskip 1truemm \langle a, z\rangle}.
\end{eqnarray*}

Next, by the reproducing kernel property, we obtain:
\begin{eqnarray*}
\left \vert \langle V_a k_z, k_w\rangle \right \vert
&=&\left \vert \langle k_{z+a} e^{-i\Im m\hskip 1truemm \langle a, z\rangle}, k_w\rangle\right \vert= \left \vert \langle k_{z+a}, k_w\rangle \right \vert\\
&=&e^{-\frac {|w|^2}2}\left \vert \langle k_{z+a}, K_w\rangle \right \vert=e^{-\frac {|w|^2}2}\left \vert k_{z+a}(w)\right \vert\\
&=&e^{-\frac {|w|^2}2}\left \vert e^{\langle w, z+a\rangle-\frac {|z+a|^2}2}\right \vert\\
&=&e^{-\frac {|z-w+a|^2}2}.
\end{eqnarray*}

We shall apply the following elementary observation: for all $\epsilon >0,$ we have 
$$\Re e\hskip 1truemm \langle z-w, a\rangle\leq \frac {|z-w|^2}{2\epsilon}+\frac {\epsilon |a|^2}2.$$
In conclusion,
\begin{eqnarray*}
-\frac {|z-w+a|^2}2
&=&-\frac {|z-w|^2}2-\frac {|a|^2}2+\Re e\hskip 1truemm \langle z-w, a\rangle\\
&\leq& -\frac {|z-w|^2}2-\frac {|a|^2}2+\frac {|z-w|^2}{2\epsilon}+\frac {\epsilon |a|^2}2\\
&=&\left (\frac 1{2\epsilon}-\frac 12\right )|z-w|^2+\left (\frac \epsilon 2-\frac 12\right )|a|^2
\end{eqnarray*}
So $\left \vert \langle V_a k_z, k_w\rangle \right \vert \leq e^{\left (\frac \epsilon 2-\frac 12\right )|a|^2}e^{\left (\frac 1{2\epsilon}-\frac 12\right )|z-w|^2}.$ 
This at once implies that $V_a$ is a XZ-sufficiently localized operator if $\epsilon  >1.$ Moreover, in this case, in view of Lemma \ref{lem32}, we have:
\begin{eqnarray*}
\int_{\mathbb C^n} \left \vert U_zV_aU_z \mathbbm 1(w)\right \vert^pd\mu (w)
&=&\int_{\mathbb C^n} \left \vert \langle V_a k_z, k_w\rangle \right \vert^p e^{\left (\frac p2-1\right )|z-w|^2}dV(w)\\
&\leq& e^{p\left (\frac \epsilon 2-\frac 12\right )|a|^2}\int_{\mathbb C^n} e^{\left (p\left (\frac 1{2\epsilon}-\frac 12\right )+\left (\frac p{2}-1\right ) \right )|z-w|^2}dV(w)\\
&=&e^{p\left (\frac \epsilon 2-\frac 12\right )|a|^2}\int_{\mathbb C^n} e^{\left (\frac p{2\epsilon}-1\right )|z-w|^2}dV(w)\\
&=&e^{p\left (\frac \epsilon 2-\frac 12\right )|a|^2}\int_{\mathbb C^n} e^{\left (\frac p{2\epsilon}-1\right )|\zeta|^2}dV(\zeta).
\end{eqnarray*}

This integral converges if $\frac p{2\epsilon}-1<0,$ that is, if $p<2\epsilon.$ Since the adjoint of $V_a$ is $V_a^\star=V_{-a},$ we conclude that 
\begin{equation*}
V_a\in \bigcap_{\epsilon >1} \bigcap_{2<p<2\epsilon} \mathcal L_p=\bigcap_{2<p<\infty} \mathcal L_p:
\end{equation*}
each operator $V_a$ is strongly localized.
\end{proof}

\subsection{A class of sufficiently localized operators which are not strongly localized}\label{ssec42}
Let $r\in [0, 1).$ We define the operator $T_r$ as follows.
$$T_r f(z)=f(-rz), \quad z\in \mathbb C$$
for any entire function in $\mathbb C.$ This is the particular case $A=-rI_n, \hskip 2truemm B=0$ ($I_n$ is the $n\times n$ identity matrix)  of the  composition operators $C_\varphi, \varphi (z)=Az+B,$ studied in Section 7 
below. The following theorem was proved by Wang, Cao and Zhu \cite[Section 3]{WCZ}. 

\begin{thm}\label{thm42}
Let $p\in (2, \infty).$ 
\begin{enumerate}
\item
Let $0<r<1.$ The operator $T_r$ belongs to $\mathcal L_p$ if $2<p\leq \frac 4{1+r}.$  It does not belong to $\mathcal L_p$ if $p>\frac 4{1+r}.$ 
\item
The operator $T_0$ belongs to $\mathcal L_p$ if $2<p\leq 4.$  It does not belong to $\mathcal L_p$ if $p>4.$
\end{enumerate}
\end{thm}

\begin{proof}
\begin{enumerate}
\item
We have:
\begin{eqnarray*}
U_zT_rU_z \mathbbm 1 (w)&=&k_z (w)\left (T_r k_z\right ) (z-w)\\
&=&k_z (w)k_z (-r(z-w))\\
&=&e^{-(r+1)|z|^2}e^{w\cdot \overline {(1+r)z}}.
\end{eqnarray*}
Then 
\begin{eqnarray*}
\int_{\mathbb C} \left \vert U_zT_rU_z \mathbbm 1 (w)\right \vert^p d\mu (z)
&=&\frac 1{\pi}e^{-p(r+1)|z|^2}\int_{\mathbb C^n} \left \vert e^{w\cdot \overline {(1+r)z}}\right \vert^pe^{-|w|^2}dV(w)\\
&=&\frac 1{\pi}e^{-p(r+1)|z|^2}\int_{\mathbb C^n} \left\vert e^{w\cdot \overline{\left(\frac {(1+r)pz}2\right)}}\right\vert^2e^{-|w|^2}dV(w)\\
&=&\frac 1{\pi}e^{-p(r+1)|z|^2}e^{\left(\frac {(1+r)p|z|}2 \right)^2}\\
&=&\frac 1{\pi}e^{-p(1+r)|z|^2\left(1-\frac{(1+r)p}4 \right)}.
\end{eqnarray*}
We obtain that  $\sup \limits_{z\in \mathbb C} \int_{\mathbb C} \left \vert U_zT_rU_z \mathbbm 1 (w)\right \vert^p d\mu (w)<\infty$ (the operator $T_r$ belongs to $\mathcal L_p)$ if and only if $1-\frac {(1+r)p}4\geq 0.$ This condition amounts to $2<p\leq \frac 4{1+r}.$ 
\item
The proof of (2) is elementary. We leave it to the reader.
\end{enumerate}
\end{proof}

\section{Other examples of localized operators. The class of Toeplitz operators with complex measure symbols}\label{sec5}
The analogous study for the Bergman space in the unit disc was performed in \cite{Z22}.
\subsection{Reminder of some measure theory}\label{ssec51}
In this subsection, for basic measure theory, we refer to  \cite[Chapter 6]{R87}.
 We start with the following definition.
\begin{defn}\label{def51}
We denote by $\mathcal B$ the Borel $\sigma$-algebra in $\mathbb C^n.$ We set $\mathcal K_l=\left \{z\in \mathbb C^n: |z| \leq l\right \}$ so that $\mathbb C^n=\cup_{l=1}^\infty \mathcal K_l.$ We call $\mathcal B_l$ the Borel algebra on $\mathcal K_l.$ Notice that $\mathcal B_l=\{A\cap \mathcal K_l: A\in \mathcal B\}.$

Let $\nu$ be a complex set-function on the Borel $\sigma$-algebra $\cup_{l=1}^\infty \mathcal B_l$ in $\mathbb C^n,$ which is such that the complex set-function $\omega$ defined on $\cup_{l=1}^\infty \mathcal B_l$ by 
\begin{equation}\label{eq51}
\omega (A)=\int_A e^{-|z|^2}d\nu (z), \quad A\in \mathcal B_l, \quad l=1, 2, \cdots,
\end{equation}
extends to a complex measure on $\mathbb B$ as follows: for every member $A$ of $\mathcal B,$
$$\omega (A)=\lim \limits_{l\rightarrow \infty}\int_{A\cap \mathcal K_l} e^{-|z|^2}d\nu (z), \quad A\in \mathcal B.$$
We denote by $|\omega|$ the total variation measure of $\omega$ on $\mathcal B,$ and we call $|\nu|$ the positive measure $|\nu|$ defined on $\mathcal B$ by 
\begin{equation}\label{eq52}
d|\nu| (z)=e^{|z|^2}d|\omega| (z), \quad z\in \mathbb C^n. 
\end{equation}
\end{defn}

We next consider the following decomposition of $\omega:$
\begin{equation}\label{dec} 
\omega=\omega^1-\omega^2+i\left (\omega^3-\omega^4\right ),
\end{equation}
where $\omega^1=\frac 12 \left (|\omega|+\Re e \hskip 1truemm \omega \right ), \hskip 1truemm \omega^2=\frac 12 \left (|\omega|-\Re e \hskip 1truemm \omega\right ), \hskip 1truemm \omega^3=\frac 12 \left (|\omega|+\Im m \hskip 1truemm \omega\right ), \hskip 1truemm \omega^4=\frac 12 \left (|\omega|-\Im m \hskip 1truemm \omega \right ).$ For all $l=1, 2, \cdots,$ each $\omega^j, j=1, 2, 3, 4,$ is a positive measure on $\mathcal B_l.$ We extend $\omega^j$ to $\mathcal B$ as follows:
\begin{equation}\label{lim} 
\omega^j (A)=\lim \limits_{l\rightarrow \infty} \omega^j (A\cap \mathcal K_l), \quad A\in \mathcal B.
\end{equation}
We prove the following lemma.
\begin{lem}\label{lem52}
The non-negative set-functions $\omega^j, \hskip 1truemm j=1, 2, 3, 4,$ are well-defined on $\mathcal B$ and define positive measures on $\mathcal B.$ Moreover, for every member $A$ of $\mathcal B,$ we have $\omega^j (A)\leq |\omega| (A), j=1, 2, 3, 4.$
\end{lem}

\begin{proof}
It is clear that each $\nu^j, \hskip 1truemm j=1, 2, 3, 4,$ is a positive measure on $\cup_{l=1}^\infty \mathcal B_l.$ Moreover, for each 
for each $j=1, 2, 3, 4,$ the sequence $\left \{\omega^j (A\cap \mathcal K_l\right \}_{l=1}^\infty$ is non-decreasing $(A\in \mathbb B)$ and 
for each $l,$ we have  $\omega^j (A\cap \mathcal K_l\leq |\omega| (A\cap \mathcal K_l).$ Since the sequence $\left \{|\omega| (A\cap \mathcal K_l) \right \}_{l=1}^\infty$ is convergent, it is bounded above and so is the sequence $\left \{\omega^j (A\cap \mathcal K_l\right \}_{l=1}^\infty.$ This implies that the sequence $\left \{\omega^j (A\cap \mathcal K_l\right \}_{l=1}^\infty$ is convergent: this justifies the existence of the limit in \eqref{lim}: each $\omega^j, \hskip 1truemm j=1, 2,\cdots,$ is a finite positive measure on $\mathbb B.$ Finally, we obtain that for every member $A$ of $\mathcal B,$ we have $\omega^j (A)\leq |\omega| (A), j=1, 2, 3, 4.$ 
\end{proof}

\subsection{Definition and preliminary results}\label{ssec52}
We begin with the following proposition.

\begin{prop}\label{pro53}
Let $\nu$ be a complex function on the Borel $\sigma$-algebra  $\cup_{l=1}^\infty \mathcal B_l$ which satisfies the hypotheses of Definition \ref{def51} and we call $\omega$ the associated complex measure on $\mathcal B.$ 
We suppose that the total variation measure $|\omega|$ of $\omega$ satisfies the condition
\begin{equation}\label{eq55}
\int_{\mathbb C^n} \left \vert K_z(w)\right \vert d|\omega| (w)<\infty
\end{equation}
for all $z\in \mathbb C^n.$
The following function $F$ is well-defined in $\mathbb C^n:$ 
$$F(z)=\int_{\mathbb C^n}  \overline{K_z (w)}f(w)e^{-|w|^2}d\nu (w), \quad z\in \mathbb C^n, \hskip 2truemm f\in \mathcal S.$$
\end{prop}

\begin{proof}
It suffices to show that
$$\int_{\mathbb C^n}  |K_z (w)||f(w)d|\omega| (w), \quad z\in \mathbb C^n, \hskip 2truemm f\in \mathcal S.$$
Since $f\in \mathcal S$ can be written $f(w)=\sum_{j=1}^J a_jk_{z_j} (w),$ it is enough to prove that
$$\int_{\mathbb C^n}  |K_z (w)||K_{z_j}(w)d|\omega| (w), \quad z, z_j\in \mathbb C^n,$$
or what amounts to the same,
$$\int_{\mathbb C^n}  |e^{w\cdot \bar{z}}e^{w\cdot \bar{z}_j}|(w)d|\omega| (w)=\int_{\mathbb C^n} |e^{w\cdot (\overline{z+z_j})}|d|\omega| (w)=\int_{\mathbb C^n} |K_{z+z_j} (w)|d|\omega| (w).$$
The required conclusion follows from \eqref{eq55}.
\end{proof}

\begin{defn}\label{def54}
Let $\nu$ be a complex function on the Borel $\sigma$-algebra $\cup_{l=1}^\infty \mathcal B_l$ which satisfies the hypotheses of the previous proposition. 
The {\textit {Toeplitz operator}} $T_\nu$ {\textit {with measure symbol}} $\nu$ is densely defined on $H^2 (\mathbb C^n, d\mu)$ as
$$T_\nu f(z)= \int_{\mathbb C^n}  \overline{K_z (w)}f(w)d\omega (w), \quad z\in \mathbb C^n, \hskip 2truemm f\in \mathcal S.$$
\end{defn}

The following lemma was proved in \cite[Lemma 2.2]{HL11}.
 We give its proof for completeness.

\begin{lem}\label{lem55}
Let $m$ be a positive measure on $\mathcal B$ and $\alpha >0, r >0, \hskip 1truemm 0<p<\infty.$ Then there exists some positive constant $C=C(r, p)$ such that
$$\int_{\mathbb C^n} \left \vert f(z)e^{-\alpha |z|^2}\right \vert^p dm(z)\leq C\int_{\mathbb C^n} \left \vert f(z)e^{-\alpha |z|^2}\right \vert^p m(B(z, r))dV(z)$$
for all entire functions in $\mathbb C^n.$
\end{lem}

\begin{proof}
Applying Lemma \ref{lem37},
 we get
$$\left \vert f(z)e^{-\alpha |z|^2}\right \vert^p\leq C\int_{B(z, r)} \left \vert f(w)e^{-\alpha |w|^2}\right \vert^p dV(w).$$
Integrating both sides with respect to $dm,$ we obtain:
\begin{eqnarray*}
\int_{\mathbb C^n} \left \vert f(z)e^{-\alpha |z|^2}\right \vert^p dm(z)
&\leq & C\int_{\mathbb C^n} \left(\int_{B(z, r)} \left \vert f(w)e^{-\alpha |w|^2}\right\vert^p dV(w)\right)dm(z)\\
&=&C\int_{\mathbb C^n} \left \vert f(w)e^{-\alpha |w|^2}\right \vert^p \left (\int_{\mathbb C^n} \chi_{B(w, r)} (z)dm(z)\right )dV(w)\\
&=&C\int_{\mathbb C^n} \left \vert f(w)e^{-\alpha |w|^2}\right \vert^p m(B(w, r))dV(w).
\end{eqnarray*}

To get the first equality, we applied the Fubini-Tonelli theorem.
\end{proof}

From this lemma, we deduce the following corollary.

\begin{cor}\label{cor56}
Let $m$ be a positive measure on $\mathcal B$ such that for some positive number $r,$ we have $m(B(\cdot, r))\in L^\infty.$ Then for every positive number $p,$  there exists some positive constant $C=C(r, p)$ such that
\begin{equation*}
\int_{\mathbb C^n} \left \vert f(z)e^{-\alpha |z|^2}\right \vert^p dm(z)\leq C\left \Vert m\right \Vert \int_{\mathbb C^n} \left \vert f(z)e^{-\alpha |z|^2}\right \vert^p dV(z)
\end{equation*}
for all entire functions in $\mathbb C^n.$ Here, $\left \Vert m\right \Vert=\left \Vert m (B(\cdot, r)\right \Vert_{L^\infty (\mathbb C^n)}.$
\end{cor}

The following definition was introduced in  \cite{Z12} 
for $n=1.$

\begin{defn}\label{def57}
A positive measure $m$ on $\mathcal B$  is called {\textit {Fock-Carleson measure}} if there exists a positive constant $C=C(m)$ such that
\begin{equation*}
\int_{\mathbb C^n} \left \vert f(z)e^{-\alpha |z|^2}\right \vert^p dm(z)\leq C\int_{\mathbb C^n} \left \vert f(z)e^{-\alpha |z|^2}\right \vert^p dV(z)
\end{equation*}
for all entire functions in $\mathbb C^n.$  
\end{defn}

\begin{rem}\label{rem58}
\begin{enumerate}
\item[1.]
The converse in Corollary \ref{cor56}
is also true  \cite[Theorem 3.1]{HL11} 
 (cf. \cite[Theorem 2.3]{IZ10} and \cite[Theorem 3.29]{Z12} for $n=1.$)
Namely, the following two assertions are equivalent.
\begin{enumerate}
\item
The measure $m$ is a Fock-Carleson measure;
\item
For some positive number $r,$ we have $m(B(\cdot, r))\in L^\infty.$
\end{enumerate}
\item[2.]
The condition  \eqref{eq55}
 is valid when $|\nu|$ is a Fock-Carleson measure (the positive measure $|\nu|$ was defined in the equation \eqref{eq52} of Definition \ref{def51}.
 Indeed, since $K_z$ is an entire function in $\mathbb C^n,$
in this case, for $\alpha=1,$ we get:
\begin{eqnarray*}
\int_{\mathbb C^n} \left \vert K_z(w)\right \vert d|\omega| (w)
&=&\int_{\mathbb C^n} \left \vert K_z(w)\right \vert e^{-|w|^2}d|\nu| (w)\\
&\leq& C\int_{\mathbb C^n} \left \vert K_z(w)\right \vert e^{-|w|^2}dV(w)\\
&= &C\int_{\mathbb C^n} \left \vert K_{\frac z2}(w)\right \vert^2 e^{-|w|^2}dV(w)\\
&=&Ce^{\left (\frac {|z|}2\right )^2}<\infty.
\end{eqnarray*}
\end{enumerate}
\end{rem}

We next prove the following proposition which generalizes  \cite[Lemma 4.1]{HL11}.

\begin{prop}\label{pro59}
Let $\nu$ be a complex function on the Borel $\sigma$-algebra  in $\mathbb C^n$ which satisfies the hypotheses of Definition \ref{def51}. 
Suppose moreover that $|\nu|$ is a Fock-Carleson measure. 
Then for all $f, g\in \mathcal S,$ we have
\begin{equation}\label{eq56}
\langle T_\nu f, g\rangle= \int_{\mathbb C^n} f(w)\overline{g(w)}d\omega (w).
\end{equation}
\end{prop}

\begin{proof}
It suffices to show that for all $u, v\in \mathbb C^n,$ we have
$$\langle T_\nu k_u, k_v\rangle=\int_{\mathbb C^n} k_u (w)\overline{k_v (w)}d\omega (w).$$
To justify an application of the Fubini theorem, we point out that
\begin{eqnarray*}
&&\int_{\mathbb C^n} \left (\int_{\mathbb C^n} |K_z (w)k_u (w)|d|\omega| (w)\right ) |k_v (z)|e^{-|z|^2}dV(z)\\
&=&
\int_{\mathbb C^n} \left (\int_{\mathbb C^n} |K_z (w)k_u (w)|e^{-|w|^2}d|\nu| (w)\right )|k_v (z)|e^{-|z|^2}dV(z)\\
&\leq & C\int_{\mathbb C^n} \left (\int_{\mathbb C^n} |K_z (w)k_u (w)|e^{-|w|^2}dV(w)\right )|k_v (z)|e^{-|z|^2}dV(z)\\
&\leq & C\int_{\mathbb C^n} \left (\int_{\mathbb C^n} |e^{\langle w, z+u\rangle}|e^{-\frac {|u|^2}2}e^{-|w|^2}dV(w)\right )|e^{\langle z, v\rangle}e^{-\frac {|v|^2}2}|e^{-|z|^2}dV(z)\\
&=&Ce^{-\frac {|u|^2+|v|^2}2}\int_{\mathbb C^n} \left (\int_{\mathbb C^n} |e^{\langle w, \frac {z+u}2\rangle}|^2e^{-|w|^2}dV(w)\right )e^{\Re e \hskip 1truemm\langle z, v\rangle}e^{-|z|^2}dV(z)\\
&=&Ce^{-\frac {|u|^2+|v|^2}2} \int_{\mathbb C^n} e^{\frac {|z+u|^2}4}e^{\Re e \hskip 1truemm\langle z, v\rangle}e^{-|z|^2}dV(z).
\end{eqnarray*}

For the first inequality, we applied the fact that $|\nu|$ is a Fock-Carleson measure. An easy calculation gives 
\begin{eqnarray*}
\frac {|z+u|^2}4+\Re e \hskip 1truemm\langle z, v\rangle-|z|^2
&=&\frac {|u|^2}4-\frac {3|z|^2}4+\Re e \hskip 1truemm\langle z, \frac u2+v\rangle\\
&=&\frac {|u|^2}4-\frac 34\left (|z|^2-\frac 43 \Re e \hskip 1truemm\langle z, \frac u2+v\rangle\right )\\
&=&\frac {|u|^2}4-\frac 34\left (\left \vert z-\frac 23\left (\frac u2+v \right ) \right \vert^2-\frac 49\left \vert \frac u2+v \right \vert^2 \right ).
\end{eqnarray*}
Whence
\begin{eqnarray}\label{eq57}
\int_{\mathbb C^n} \left (\int_{\mathbb C^n} |K_z (w)k_u (w)|d|\nu| (w)\right )|k_v (z)|e^{-|z|^2}dV(z)
&\leq& C(u, v)\int_{\mathbb C^n} e^{-\frac 34\left \vert z-\frac 23\left (\frac u2+v \right ) \right \vert^2}dV(z) \nonumber\\
&=&C(u, v)\int_{\mathbb C^n} e^{-\frac 34\left \vert \zeta \right \vert^2}dV(\zeta)<\infty. \label{Fub}
\end{eqnarray}
Now, we obtain:
\begin{eqnarray}
\langle T_\nu k_u, k_v\rangle
&=&\int_{\mathbb C^n} \left (\int_{\mathbb C^n} \overline{K_z (w)}k_u (w)e^{-|w|^2}d\nu (w)\right )\overline{k_v (z)}d\mu (z) \nonumber\\
&=& \int_{\mathbb C^n} \overline {\left (\int_{\mathbb C^n} K_z (w)k_v (z)d\mu (z)\right )}k_u (w)e^{-|w|^2}d\nu (w) \label{toep1}\\ 
&=& \int_{\mathbb C^n} k_u(w)\overline{k_v(w)}e^{-|w|^2}d\nu (w) \label{toep2}.\label{eq59}
\end{eqnarray}
The  equality \eqref{toep1} 
follows from an application of the Fubini theorem, while the equality \eqref{toep2} 
 follows from an application of the reproducing kernel property.
\end{proof}

From the previous lemma, we deduce the following corollary.

\begin{cor}\label{cor510}
Let $\nu$ be a complex function on the Borel $\sigma$-algebra $\cup_{l=1}^\infty \mathcal B_l$ in $\mathbb C^n,$ which satisfies the hypotheses of Definition \ref{def51}. 
Suppose moreover that $|\nu|$ is a Fock-Carleson measure. Then the Toeplitz operator $T_\nu$ is an admissible operator.
\end{cor}

\begin{proof}
We denote by $\overline{\nu}$ the complex function on the Borel $\sigma$-algebra  in $\mathbb C^n$ defined by $\overline{\nu} (A)=\overline{\nu (A)}$ for every Borel set $A$ ($\overline{\nu}$ is the complex-conjugate measure of the measure $\nu).$ It is elementary to check that $\overline{\nu}$ is a complex function on the Borel $\sigma$-algebra $\cup_{l=1}^\infty \mathcal B_l$ in $\mathbb C^n,$ which satisfies the hypotheses of Definition \ref{def51}.
 Moreover, $|\overline 
\nu|=|\nu|$ and so  $|\overline{\nu}|$ is also a Fock-Carleson measure. Hence, by assertion 2 of Remark \ref{rem58}, the measure $\overline{\nu}$ satisfies the condition \eqref{eq55}
 and then the function $T_{\overline \nu} g$ is well-defined for all $g\in \mathcal S.$ Next, Proposition \ref{pro59}
 applies to $\overline{\nu}$ so that for all $g, f\in \mathcal S,$ we have
\begin{eqnarray*}
\langle T_{\overline \nu} g, f\rangle
&=& \int_{\mathbb C^n} g(w)\overline{f(w)}e^{-|w|^2}d\overline \nu (w)\\
&=&\overline{\int_{\mathbb C^n} f(w)\overline{g(w)}e^{-|w|^2}d\nu (w)}\\
&=&\overline{\langle T_{\nu} f, g\rangle}.
\end{eqnarray*}
We have shown that $\left (T_\nu\right )^\star=T_{\overline{\nu}}$ in the relation \eqref{eq21} of an admissible operator. 
\end{proof}

In the rest of the section, we adapt the argument of Zorboska \cite{Z22}
for the Bergman space on the unit disc. We first prove the following lemma.

\begin{lem}\label{lem511}
Let $\nu$ be a complex function on the Borel $\sigma$-algebra $\cup_{l=1}^\infty \mathcal B_l$ in $\mathbb C^n,$ which satisfies the hypotheses of Definition \ref{def51} 
 and satisfies the condition \eqref{eq55}.
 For all $z\in \mathbb C^n,$ we have the following identity
\begin{equation}\label{Zo}
U_zT_\nu U_z \mathbbm 1=T_{\nu \circ \varphi_z} \mathbbm 1,
\end{equation}
where  
\begin{eqnarray*}
\varphi_z:&\mathbb C^n\,\rightarrow&\mathbb C^n\\
&w\,\mapsto&z-w
\end{eqnarray*}

and $\nu \circ \varphi_z$ is the image measure of $\nu$ under the mapping $\varphi_z.$
\end{lem}

\begin{proof}
Let $w\in \mathbb C^n.$ On the one hand, we have 
$$U_z \mathbbm 1 (w)=k_z (w),$$
$$ T_\nu U_z \mathbbm 1 (w)= \int_{\mathbb C^n} \overline{K_w (\zeta)}k_z (\zeta)e^{-|\zeta|^2}d\nu (\zeta)$$ 
and
$$U_zT_\nu U_z \mathbbm 1 (w)=k_z(w) \int_{\mathbb C^n} \overline{K_{z-w} (\zeta)}k_z (\zeta)e^{-|\zeta|^2}d\nu (\zeta).$$
On the other hand, we have
\begin{eqnarray*}
T_{\nu \circ \varphi_z} \mathbbm 1 (w)&= &\int_{\mathbb C^n} \overline{K_w (\zeta)}\mathbbm 1 (\zeta)e^{-|\zeta|^2}d\left (\nu \circ \varphi_z\right ) (\zeta)\\
&=&\int_{\mathbb C^n} \overline{K_w (z-\zeta)}e^{-|z-\zeta|^2}d\nu (\zeta)
\end{eqnarray*}
by the definition of the image measure $\nu \circ \varphi_z.$ It now suffices to show that for all $w, z, \zeta \in \mathbb C^n,$ the following equality holds:
\begin{equation}\label{eq511}
k_z(w)\overline{K_{z-w} (\zeta)}k_z (\zeta)e^{-|\zeta|^2}=\overline{K_w (z-\zeta)}e^{-|z-\zeta|^2}.
\end{equation}
The left-hand side of \eqref{eq511} is equal to
$$e^{\langle w, z\rangle}e^{-\frac {|z|^2}2}e^{\langle z-w, \zeta \rangle}e^{\langle \zeta, z\rangle}e^{-\frac {|z|^2}2}e^{-|\zeta|^2}=e^{\langle w, z-\zeta\rangle}e^{-|z|^2}e^{2\Re e \hskip 1truemm\langle z, \zeta\rangle}e^{-|\zeta|^2}=e^{\langle w, z-\zeta\rangle}e^{-|z-\zeta|^2},$$
while its right-hand side is equal to
$$e^{\langle w, z-\zeta\rangle}e^{-|z-\zeta|^2}.$$
This finishes the proof of the lemma.
\end{proof}

\subsection{The main result}\label{ssec53}
We record the following duality result for Fock spaces  (\cite[Theorem 7.4]{JPZ17} for $n\geq1$, \cite[Theorem 2.23]{Z12} for $n=1.$)

\begin{thm}\label{thm512}
Suppose $\alpha, \beta >0, \hskip 1truemm 1<p<\infty.$ Then the topological dual space of $F^p_\alpha$ can be identified with $F^{p'}_\beta$ under the integral pairing
$$\langle f, g\rangle_\gamma=\lim \limits_{l\rightarrow \infty} \frac \gamma \pi \int_{\mathcal K_l} f(z)\overline{g(z)}e^{-\gamma |z|^2}dV(z),$$
where $\gamma=\sqrt {\alpha \beta}.$
\end{thm}

Our result for Toeplitz operators is the following.

\begin{thm}\label{thm513}
Let $\nu$ be a complex function on the Borel $\sigma$-algebra $\cup_{l=1}^\infty \mathcal B_l$ in $\mathbb C^n,$ which satisfies the hypotheses of Definition \ref{def51}.
Suppose moreover that the measure $|\nu|$ is a Fock-Carleson measure. Then the associated Toeplitz operator $T_\nu$ is a sufficiently localized operator. More precisely, it belongs to $\mathcal L_p, \hskip 1truemm 2<p<4.$
\end{thm}

\begin{proof}
In view of Lemma \ref{lem511},
 the required result is equivalent to the following estimate.
$$\sup \limits_{z\in \mathbb C^n} \int_{\mathbb C^n} \left \vert T_{\nu \circ \varphi_z}\mathbbm 1 (w)\right \vert^pd\mu(z)<\infty$$
for all $p\in (2, 4).$ Since $\mathbbm 1\in F^p_{2\left (1-\frac p4\right )},$ it suffices to show that the family of Toeplitz operators $\left \{T_{\nu \circ \varphi_z}\right \}_{z\in \mathbb C^n}$ is uniformly bounded from $F^p_{2\left (1-\frac p4\right )}$ to $F^p_{\frac 2p}.$

We shall rely on the following theorem  \cite[Theorem 4.2]{HL11}. 

\begin{thm}\label{thm514}
Let $p\in (1, \infty)$ and let $m$ be a positive measure in $\mathbb C^n$ satisfying the estimate
\begin{equation}\label{eq512}
\int_{\mathbb C^n} \left \vert K_z(w)\right \vert e^{-|w|^2}dm (w)<\infty
\end{equation}
for all $z\in \mathbb C^n.$ The following two statements are equivalent.
\begin{enumerate}
\item
$T_m: F^p_{2\left (1-\frac p4\right )} \rightarrow F^p_{\frac 2p}$ is bounded;
\item
$m(B(\cdot, r))\in L^\infty (\mathbb C^n)$ for some (all) $r>0.$
\end{enumerate}
Moreover, there exists an absolute positive constant $C$ independent of $m$ such that
$$\left \Vert T_m f \right \Vert_{F^p_{\frac 2p}}\leq C\left \Vert m\right \Vert^{\frac 1p} \left \Vert f\right \Vert_{F^p_{2\left (1-\frac p4\right )} }$$
for all $f\in F^p_{2\left (1-\frac p4\right )}.$ 
\end{thm}

Let $\nu$ be a complex function on the Borel $\sigma$-algebra $\cup_{l=1}^\infty \mathcal B_l$ in $\mathbb C^n,$ which satisfies the hypotheses of Definition \ref{def51}. 
According to \eqref{dec}, 
 we have the decomposition $\nu=\nu^1-\nu^2+i(\nu^3-\nu^4),$ where each $\nu^j, j=1, 2, 3, 4,$ is a positive measure on $\mathcal B$ such that $\nu^j \leq |\nu|.$  

We recall that $\nu^j \circ \varphi_z (B(w, r))=\nu^j (B(z-w, r))$ for all $z, w\in \mathbb C^n.$ Then for all $z\in \mathbb C^n,$ each positive measure $\nu^j$ satisfies $\nu^j \circ \varphi_z (B(\cdot, r))\in L^\infty (\mathbb C^n)$ with an $L^\infty (\mathbb C^n)$-norm smaller than that of $|\nu| (B(\cdot, r))=\int_{B(\cdot, r)}d|\nu| (w).$  

Suppose further that the total variation measure $|\nu|$ is a Fock-Carleson measure. Applying Theorem \ref{thm514},
 we obtain that each $T_{\nu^j\circ \varphi_z}: F^p_{2\left (1-\frac p4\right )}\rightarrow F^p_{\frac 2p}$ is uniformly bounded $(j=1, 2, 3, 4)$ and since $T_{\nu\circ \varphi_z}=(T_{\nu^1\circ \varphi_z}-T_{\nu^2\circ \varphi_z})
+i(T_{\nu^3\circ \varphi_z}-T_{\nu^4\circ \varphi_z}),$ the family $\left \{T_{\nu\circ \varphi_z} \right \}_{z\in \mathbb C^n}$ is uniformly bounded from $F^p_{2\left (1-\frac p4\right )}$ to $F^p_{\frac 2p}.$ 
\end{proof}

Theorem \ref{thm514}
 applies to the following two particular cases: 1) $d\nu=fdV,$ where $f$ is a bounded measurable function on $\mathbb C^n$ (in this case, $T_\nu=T_f$ is a 'Toeplitz operator' in the sense of the Introduction); 2) $\nu$ is a Fock-Carleson (positive) measure on $\mathcal B.$ Indeed, for case 1), 
we have
$$\int_{B(z, r)} |f|dV\leq \left \Vert f\right \Vert_{L^\infty (\mathbb C^n)}V(B(z, r)=c_n\left \Vert f\right \Vert_{L^\infty (\mathbb C^n)}r^{2n}.$$
For case 2), we can also apply Theorem \ref{thm514}.

Comparing with the case of the Bergman space on the unit disc, we ask the following question. Under the assumptions of Theorem \ref{thm514},
 is the Toeplitz operator $T_\nu$ a strongly localized operator?

\section{A bounded linear operator on $H^2 (\mathbb C, d\mu)$ which is not a member of the Toeplitz algebra. Compact linear operators of the same type}\label{sec6}
In this section, we restrict to $n=1$ and we study the analogue of an operator which was studied by Axler-Zheng \cite{AZ98} 
for the Bergman space on the unit disc.

\subsection{A bounded linear operator on $H^2 (\mathbb C, d\mu)$ which is not a member of the Toeplitz algebra}\label{ssec61}
It is well-known \cite[Proposition 2.1]{Z12}
 that the sequence $\{e_m\}_{m=0}^\infty$ given by $e_m (z)=\frac {z^m}{\sqrt {m!}}$ is an orthonormal basis for the Hilbert space $H^2 (\mathbb C, d\mu).$ We consider the linear operator $T$ defined on $H^2 (\mathbb C, d\mu)$ by
\begin{equation}\label{eq61}
T\left (\sum_{m=0}^\infty c_me_m\right )=\sum_{m=0}^\infty c_{2^m}e_{2^m}.
\end{equation}
It is easy to check that $T$ is bounded on $H^2 (\mathbb C, d\mu).$ Now, according to the combination of  Theorem \ref{thm11} and Theorem \ref{thm12},
 to obtain that $T$ is not weakly localized, it suffices to show that $T$ is not compact and its Berezin transform vanishes at infinity. According to Theorem \ref{thm11}, $T$ is not even a member of the Toeplitz algebra $\mathcal T^1.$

We first prove that $T$ is not compact. The sequence $\{e_m\}_{m=0}^\infty$ converges weakly to zero on $H^2 (\mathbb C, d\mu).$ Indeed, for $g=\sum_{m=0}^\infty c_mz^m\in H^2 (\mathbb C, d\mu),$ we have $\sum_{m=0}^\infty \left \vert c_m\right \vert^2<\infty$ and so $\{c_m\}\rightarrow 0.$ Hence, $\langle e_m, g\rangle=c_m\rightarrow 0.$ Suppose that $T$ is compact; then the sequence $\{\left \Vert Te_m \right \Vert\}$ would converge to zero. However, we have:
$$\left \Vert Te_m-Te_j\right \Vert=\left \Vert e_{2^m}-e_{2^j}\right \Vert=2;$$
so, the sequence $\{\left \Vert Te_m \right \Vert\}$ would not be a Cauchy sequence on $H^2 (\mathbb C, d\mu)$ and hence, it would not converge.

We next show that the Berezin transform $\widetilde T$ of $T$ vanishes at infinity. Since
 $$k_z=e^{-\frac {|z|^2}2}\sum_{m=0}^\infty \frac {\bar{z}^m}{\sqrt {m!}}e_m,$$
  we have $Tk_z=e^{-\frac {|z|^2}2}\sum_{m=0}^\infty \frac {\bar{z}^{2^m}}{\sqrt {(2^m)!}}e_{2^m}$ and hence:
$$\widetilde T (z)=e^{-|z|^2}\sum_{m=0}^\infty \left \vert \frac {z^{2^m}}{\sqrt {(2^m)!}}\right \vert^2\left \Vert e_{2^m}\right \Vert^2=e^{-|z|^2}\sum_{m=0}^\infty \frac {|z|^{2^{m+1}}}{\left (2^m\right )!}.$$
Applying the change of variable $t=|z|^2,$ the required conclusion follows from the next proposition.

\begin{prop}\label{pro61}
The following identity is valid.
\begin{equation}\label{eq62}
\lim \limits_{t\rightarrow \infty} \sum_{m=0}^\infty  \frac {t^{2^m}}{(2^m)!}e^{-t}=0.
\end{equation}
\end{prop}
\begin{proof}
For a function $f\in L^1 (\mathbb R),$ we denote by $\widehat{f}$ its Fourier transform defined by
$$\widehat{f} (\xi)=\int_{-\infty}^\infty f(t)e^{-it\xi}dt, \quad \xi\in \mathbb R.$$
We shall use the Riemann-Lebesgue lemma.

\begin{lem}\cite[Proposition 2.2.17]{G08}
For a function $f$ in $L^1 (\mathbb R),$ we have that
$$|\widehat f(t)|\rightarrow 0 \quad {\rm{as}} \quad |t|\rightarrow \infty.$$
\end{lem}

We shall also use the following lemma.

\begin{lem}\label{lem63}
For each positive integer $m,$ the function $\varphi (t)=t^me^{-t}\chi_{(0, \infty} (t)$ is the Fourier transform of the function $f(x)=\frac {m!}{(1-ix)^{m+1}}.$
\end{lem}

\begin{proof}
Using the change of variable $\tau=(1+i\xi)t,$ Cauchy's Theorem for holomorphic functions and the Gamma integral, it is easy to prove that $\widehat \varphi (\xi)=\frac {m!}{(1+i\xi)^{m+1}}.$ We have $\varphi, \widehat \varphi \in L^1 (\mathbb R).$ So by the Fourier inversion on $L^1 (\mathbb R)$  (cf. e.g. \cite[Exercise 2.2.6]{G08}),
 we have $\left (\widehat \varphi \right )^\vee=\varphi$ a. e. where $\vee$ stands for
$$\psi^\vee (t)=\int_{-\infty}^\infty \psi (\xi)e^{i\xi t}d\xi \quad (\psi \in L^1 (\mathbb R), \hskip 1truemm t\in \mathbb R).$$ 
In our case, we get
$$\varphi (t)=\left (\widehat \varphi \right )^\vee (t)=\int_{-\infty}^\infty \frac {m!}{(1+i\xi)^{m+1}}e^{i\xi t}d\xi \hskip 2truemm a. e.$$
and taking the conjugate, we obtain:
$$\overline{\varphi (t)}=\varphi (t)=\int_{-\infty}^\infty \frac {m!}{(1-i\xi)^{m+1}}e^{-i\xi t}d\xi \hskip 2truemm a. e.$$
In other words, $\varphi=\widehat f$ with $f(x)=\frac {m!}{(1-ix)^{m+1}}.$
\end{proof}

It follows from the previous lemma that the function $F(t)=\sum_{m=0}^\infty  \frac {t^{2^m}}{(2^m)!}e^{-t}\chi_{(0, \infty} (t)$ is the Fourier transform of the function $\Phi (x)=\sum_{m=0}^\infty \frac 1{(1-ix)^{2^m+1}}.$ 
We now show that $\Phi\in L^1 (\mathbb R),$ i.e. 
$$I:=\int_{-\infty}^\infty \left \vert \sum_{m=0}^\infty \frac 1{(1-ix)^{2^m+1}}\right \vert dx < \infty.$$
We have
\begin{eqnarray*}
I&\leq& \sum_{m=0}^\infty \int_{-\infty}^\infty \frac 1{(x^2+1)^{\frac {2^m+1}2}}dx\\
&=&2\sum_{m=0}^\infty \int_0^\infty \frac 1{(x^2+1)^{\frac {2^m+1}2}}dx\\
&=&2\left \{\int_0^\infty \frac 1{x^2+1}dx+\int_0^\infty \frac 1{(x^2+1)^{\frac 32}}dx+\sum_{m=2}^\infty \int_0^\infty \frac 1{(x^2+1)^{\frac {2^m+1}2}}dx      \right \}.
\end{eqnarray*}

It suffices to prove that 
\begin{equation}\label{sum}
\sum_{m=2}^\infty \int_0^\infty \frac 1{(x^2+1)^{\frac {2^m+1}2}}dx< \infty.
\end{equation}
For each $m,$ we have
\begin{equation}\label{eq64}
\int_0^\infty \frac 1{(x^2+1)^{\frac {2^m+1}2}}dx\leq \int_0^\infty \frac 1{(x^2+1)^{2^{m-1}}}dx.
\end{equation}
Referring to \cite[Formula 3.249]{GR07},
 we shall use the following formula:
\begin{eqnarray*}
\int_0^\infty \frac 1{(x^2+1)^{2^{m-1}}}dx&=&\frac {\pi\times 1\times 3\times 5\times \cdots \times (2^m-3)}{2\times 2\times 4\times \cdots \times (2^m-2)}
\end{eqnarray*}
To conclude, we shall show that $\sum_{m=2}^\infty \frac {1\times 3\times 5\times \cdots \times (2^m-3)}{2\times 4\times \cdots \times (2^m-2)}< \infty.$
For a positive integer $p\geq 3,$ we have
\begin{eqnarray*}
\frac {1\times 3\times 5\times \cdots \times (2p-1)}{2\times 4\times \cdots \times (2p)}&=&\frac {(2p-1)!}{2^2\times 4^2\times \cdots (2p-2)^2\times (2p)}\\
&=&\frac {(2p-1)!}{2^{2(p-1)}\times 1^2\times 2^2\times \cdots (p-1)^2 \times (2p)}\\
&=&\frac {(2p-1)!}{2^{2(p-1)}\times [(p-1)!]^2 \times (2p)}.
\end{eqnarray*}
According to Stirling's formula, when $p\rightarrow \infty,$ we have the equivalence:
\begin{eqnarray*}
\frac {(2p-1)!}{2^{2(p-1)}\times [(p-1)!]^2 \times (2p)}
&\sim& \frac {\left (\frac {2p-1}e \right )^{2p-1}\sqrt {2\pi (2p-1)}}{2^{2(p-1)}\times \left (\frac {p-1}e \right )^{2(p-1)}\times 2\pi (p-1) \times (2p)}\\
&=&\frac {2^2}{2^{2p}}\times \frac {\left (\frac {p-1}e \right )^2}{\frac {2p-1}e} \times \left (\frac {\frac {2p-1}e}{\frac {p-1}e} \right )^{2p}\times \frac {\sqrt {2\pi (2p-1)}}{2\pi (p-1)\times (2p)}\\
&=&2^2\times \frac {\left (\frac {p-1}e \right )^2}{\frac {2p-1}e}\times \left (\frac {\frac {2p-1}e}{\frac {2(p-1)}e} \right )^{2p}\times \frac {\sqrt {2\pi (2p-1)}}{2\pi (p-1)\times (2p)}
\end{eqnarray*}
Since $\lim \limits_{p\rightarrow \infty} \left (\frac {2p-1}{2(p-1)} \right )^{2p}=e,$ we obtain:
\begin{eqnarray*}
\frac {(2p-1)!}{2^{2(p-1)}\times [(p-1)!]^2 \times (2p)}&\sim 2^2\frac {p^2}{2pe}\times e \times \frac 1{\sqrt {2\pi}}\times \frac {\sqrt {2p}}{p\sqrt {2p}} = \frac 1{\sqrt {\pi p}}.
\end{eqnarray*}
We take $p=2^{m-1}-1;$ we deduce that when $m\rightarrow \infty,$ we have the equivalence:
\begin{equation}\label{eq65}
\sum_{m=2}^\infty \int_0^\infty \frac 1{(x^2+1)^{2^{m-1}}}dx\sim \frac 1{\sqrt \pi}\sum_{m=2}^\infty \frac 1{\left (2^{m-1}-1\right )^{\frac 12}}<\infty.
\end{equation}
This finishes the proof of \eqref{sum}. 
 The proposition is entirely proved.
\end{proof}

\subsection{Compact linear operators of the same type}\label{ssec62}
Let $\gamma=\{\gamma_m\}_{m=0}^\infty$ be a bounded sequence of complex numbers. We consider the bounded linear operator $T=T_\gamma$ on $H^2 (\mathbb C^n, d\mu)$ defined as
\begin{equation*}
T\left (\sum_{m=0}^\infty c_me_m\right )=\sum_{m=0}^\infty \gamma_m c_{2^m}e_{2^m}.
\end{equation*}
This operator is isomorphic to the bounded operator
\begin{eqnarray*} 
\widehat{T}:&\ell^2&\rightarrow\ell^2 \\
&\sum_{m=0}^\infty c_m\widehat{e}_m&\mapsto \sum_{m=0}^\infty \gamma_m c_{2^m}\widehat{e}_{2^m},
\end{eqnarray*} 
where $\{\widehat{e}_m\}_{m=0}^\infty$ is the canonical basis of $\ell^2.$ It is well-known (cf. e.g. \cite[Exercise 6.1]{B11})
 that $\widehat{T}$ is compact on $\ell^2$ if and only if the sequence $\{\gamma_m\}_{m=0}^\infty$ converges to zero. Under this condition, $T$ is also compact on $H^2 (\mathbb C^n, d\mu);$ it then follows easily from Proposition \ref{pro61}
  that the Berezin transform $\widetilde{T} (z)$ vanishes at infinity. This could also deduced from Theorem \ref{thm11}
   as well as the membership of $T$ in the Toeplitz algebra $\mathcal T^1.$ 

We end this subsection with two questions. Let $\gamma=\{\gamma_m\}_{m=0}^\infty$ be a  sequence of complex numbers which converges to zero.
\begin{enumerate}
\item
Is the associated operator $T_\gamma$ weakly localized?
\item
Is it XZ-sufficiently localized?
\end{enumerate}

\section{The class of composition operators on the Fock space of $\mathbb C^n$}\label{sec7}
\subsection{Definitions and previous results}\label{ssec71}
In this section, we consider a general class which contains the class of operators $\{T_r\}_{0\leq r< 1}$ studied in Subsection \ref{ssec42}: the class of composition operators of the Fock space $H^2 (\mathbb C^n, d\mu).$ 

We recall the following definition.

\begin{defn}\label{def71}
Given a holomorphic mapping $\varphi: \mathbb C^n \rightarrow \mathbb C^n,$ the composition operator $C_{\varphi}$ is defined on $H^2 (\mathbb C^n, d\mu)$ as 
$$C_{\varphi} (f)=f\circ \varphi, \quad f\in H^2 (\mathbb C^n, d\mu).$$
\end{defn} 

We  reiterate the question of Sadeghi \cite{S21} at the end of his PhD dissertation for the Bergman space $L^2_a (\mathcal D, dA)$ on the unit disc $\mathcal D.$ Which composition operators are weakly (strongly) localized on $L^2_a (\mathcal D, dA)?$ ($dA$ is the Lebesgue area measure on $\mathbb D$). This time, the question is for the Fock space in the complex space $\mathbb C^n:$ we completely answer this question in this setting. 

We record the following two theorems \cite[Theorem 1 and Theorem 2]{CMS03}.

\begin{thm}\label{thm72}
Suppose $\varphi: \mathbb C^n \rightarrow \mathbb C^n$ is a holomorphic mapping.
\begin{enumerate}
\item
If $C_{\varphi}$ is bounded on $H^2 (\mathbb C^n, d\mu),$ then $\varphi (z)=Az+B,$ where $A$ is an $n\times n$ matrix and $B$ is an $n\times 1$ vector. Furthermore, $\left \Vert A\right \Vert \leq 1,$ and if $|A\zeta|=|\zeta|$ for some $\zeta \in \mathbb C^n,$ then $\langle A\zeta, B\rangle =0.$
\item
Conversely, suppose that $\varphi (z)=Az+B,$ where $A$ is an $n\times n$ matrix with $\left \Vert A\right \Vert \leq 1,$ and $B$ is an $n\times 1$ vector. If $\langle A\zeta, B\rangle =0$ whenever $|A\zeta|=|\zeta|,$ then $C_{\varphi}$ is bounded on $H^2 (\mathbb C^n, d\mu).$
\end{enumerate}
\end{thm}

\begin{thm}\label{thm73}
Suppose $\varphi: \mathbb C^n \rightarrow \mathbb C^n$ is a holomorphic mapping.
\begin{enumerate}
\item
If $C_{\varphi}$ is compact on $H^2 (\mathbb C^n, d\mu),$ then $\varphi (z)=Az+B,$ where  $\left \Vert A\right \Vert < 1.$
\item
Conversely, suppose that $\varphi (z)=Az+B,$ where $A$ is an $n\times n$ matrix with $\left \Vert A\right \Vert \leq 1,$ and $B$ is an $n\times 1$ vector. If $\left \Vert A\right \Vert < 1,$ then $C_{\varphi}$ is compact on $H^2 (\mathbb C^n, d\mu).$ 
\end{enumerate}
\end{thm}

\subsection{The Berezin transform of a bounded composition operator}\label{ssec72}
According to Theorem \ref{thm11},
 if $C_{\varphi}$ is compact, then its Berezin transform vanishes at infinity. Otherwise, if $C_{\varphi}$ is not compact and its Berezin transform vanishes at infinity, then $C_{\varphi}$ is not a weakly localized operator;  this follows by contradiction  from a combination of Theorem \ref{thm11} 
 and Theorem \ref{thm12}
  According to Theorem \ref{thm11}
  , $C_{\varphi}$ does not even belong to the Toeplitz algebra $\mathcal T^1.$ It will therefore be worthy to determine those bounded, non-compact composition operators whose Berezin transform does not vanish at infinity. We prove the following lemma.

\begin{lem}\label{lem74}
Suppose that $\varphi (z)=Az+B,$ where $A$ is an $n\times n$ matrix with $\left \Vert A\right \Vert =1,$ and $B$ is an $n\times 1$ vector. We suppose further that $\langle A\zeta, B\rangle =0$ whenever $|A\zeta|=|\zeta|.$
The following two statements are equivalent.
\begin{enumerate} 
\item
The Berezin transform $\widetilde{C_{\varphi}}$ of $C_{\varphi}$ vanishes at infinity;
\item
The following equality holds:
$$\lim \limits_{|z|\rightarrow \infty} \left \{|z|^2-\Re e \hskip 1truemm \left (\langle Az, z\rangle +\langle B, z\rangle\right )\right \}=\infty.$$
\end{enumerate}
\end{lem}
\begin{proof}
For all $z\in \mathbb C^n,$ we have $C_{\varphi} k_z (w)=e^{-\frac {|z|^2}2}e^{\langle Aw+B, z\rangle}.$ Then:
\begin{eqnarray*}
\widetilde{C_{\varphi}} (z)
&=&\langle C_{\varphi} k_z, k_z\rangle\\
&=&e^{-|z|^2}\int_{\mathbb C^n} e^{\langle Aw+B, z\rangle}\overline{K_z (w)}e^{-|w|^2}dV(w)\\
&=&e^{-|z|^2}e^{\langle Az+B, z\rangle}.
\end{eqnarray*}
For the latest equality, we applied the reproducing kernel property. This implies that $\left \vert \widetilde{C_{\varphi}} (z)\right \vert=e^{-|z|^2+\Re e\hskip 1truemm \langle Az+B, z\rangle}.$
\end{proof} 

\subsection{The case $n=1$}\label{ssec73}
We first restrict to the one-dimensional case $n=1.$ In this case, Theorem \ref{thm72}
 simply says that if $C_{\varphi}$ is bounded on $H^2 (\mathbb C^n, d\mu),$ then $\varphi (z)=az+b,$ where $a$ and $b$ are complex numbers with $|a|\leq 1,$ and if $|a|=1,$ then $b=0.$ In view of Lemma \ref{lem74}
 , the Berezin transform $\widetilde{C_{\varphi}}$ of $C_{\varphi}$ does not vanish at infinity if and only if $|z|^2(1-\Re e\hskip 1truemm a)$ does not converge to $\infty$ as $|z|\rightarrow \infty.$ This is the case if and only if $\Re e\hskip 1truemm a=1$ or equivalently, if and only if $a=1.$

Next, by Theorem \ref{thm73},
 if $|a|< 1,$ the operator $C_{\varphi}$ is compact on $H^2 (\mathbb C^n, d\mu);$ then by Theorem \ref{thm11}
 , $C_{\varphi}$ is a member of $\mathcal T^1.$ Again, by Theorem \ref{thm73},
  if $|a|=1,$ $C_{\varphi}$ is not compact on on $H^2 (\mathbb C^n, d\mu);$ in this case, Bauer, Fulsche and Rodriguez \cite[Proposition 3.8]{BFR24}
   proved that $C_{\varphi}$ is a member of $\mathcal T^1$ if and only if $a=1.$

The question now is whether the composition operators $C_{\varphi}, \hskip 2truemm \varphi (z)=az+b,$ with $|a|< 1,$ or $(a, b)=(1, 0),$ are localized. The answer is positive as the following proposition asserts.

\begin{prop}\label{pro75}
Suppose that $\varphi (z)=az+b,$ where $a$ and $b$ are complex numbers.
\begin{enumerate}
\item
If $|a|< 1,$ or $(a, b)=(1, 0),$ then the associated composition operator $C_{\varphi}$ is sufficiently localized.
\item
If $|a|=1, b=0$ and $a\neq 1,$ then the associated composition operator $C_{\varphi}$ is not weakly localized. It is not even a member of the Toeplitz algebra $\mathcal T^1.$
\end{enumerate}
\end{prop}

\begin{proof}
We have:
\begin{eqnarray*}
\langle C_{\varphi}k_z, k_w\rangle
&=&e^{-\frac {|z|^2+|w|^2}2}\langle C_{\varphi}K_z, K_w\rangle\\
&=&e^{-\frac {|z|^2+|w|^2}2}C_{\varphi}K_z (w)\\
&=&e^{-\frac {|z|^2+|w|^2}2}K_z (aw+b)\\
&=&e^{-\frac {|z|^2+|w|^2}2}e^{(aw+b)\cdot \bar z}.
\end{eqnarray*}
So
\begin{eqnarray*}
&&\int_{\mathbb C} \left \vert \langle C_{\varphi}k_z, k_w\rangle\right \vert^p e^{\left (\frac p2-1)|z-w|^2 \right )}dV(w)\\
&=&e^{-\frac {p|z|^2}2}e^{\Re e\hskip 1truemm (pb\bar z)}\int_{\mathbb C} e^{p\Re e\hskip 1truemm (paw\bar z)}e^{-\frac {p|w|^2}2}e^{\left (\frac p2-1 \right )|z-w|^2}dV(w)\\
&=&e^{-\frac {p|z|^2}2}e^{\Re e\hskip 1truemm (pb\bar z)}\int_{\mathbb C} e^{p\Re e\hskip 1truemm (paw\bar z)}e^{-\frac {p|w|^2}2}e^{\left (\frac p2-1\right )\left \{|z|^2+|w|^2 -2\Re e\hskip 1truemm (w\bar z) \right \}}dV(w)\\
&=&e^{-|z|^2}e^{\Re e\hskip 1truemm (pb\bar z)}\int_{\mathbb C} e^{-|w|^2}e^{\Re e\hskip 1truemm \{w\bar z(pa -p+2)\}}dV(w)\\
&=&e^{-|z|^2}e^{\Re e\hskip 1truemm (pb\bar z)}\int_{\mathbb C} e^{-\left \vert w-\frac {(pa-p+2)z}2\right \vert^2+\frac {|pa-p+2|^2|z|^2}4}dV(w)\\
&=&e^{-|z|^2\{1-\frac {|pa-p+2|^2}4\}}e^{\Re e\hskip 1truemm (pb\bar z)}\int_{\mathbb C} e^{-|\zeta|^2}dV(\zeta).
\end{eqnarray*}

It next suffices to show that $\sup \limits_{z\in \mathbb C} e^{-|z|^2\{1-\frac {|pa-p+2|^2}4\}}e^{\Re e\hskip 1truemm (pb\bar z)}<\infty$ for some positive number $p>2.$ 

If $|a|< 1,$ it suffices to prove that $1-\frac {|pa-p+2|^2}4>0;$ this inequality is equivalent to the inequality $p<\frac {4(1-\Re e\hskip 1truemm a)}{|1-a|^2}.$ We easily check that $\frac {4(1-\Re e\hskip 1truemm a)}{|1-a|^2}>2$ whenever $|a|< 1.$ 

We have thus proved that $C_{\varphi}\in \mathcal L_p$ for all $p\in \left (2, \frac {4(1-\Re e\hskip 1truemm a)}{|1-a|^2}\right ).$

If $(a, b)=(1, 0),$ then $C_{\varphi}$ is the identity operator and it is well-known that it is strongly localized ($C_{\varphi}=V_0$ in Subsection \ref{ssec41}).
\end{proof}

\subsection{The case $n\geq 2$}\label{ssec74}
We next consider the case where $n\geq 2.$ We have

\begin{eqnarray*}
\langle C_\varphi k_z, k_w\rangle&=e^{-\frac {|z|^2+|w|^2}2}e^{\langle Aw+B, z\rangle}\\
&=e^{-\frac {|z|^2+|w|^2}2}e^{\langle w, A^\star z\rangle+\langle B, z\rangle}.
\end{eqnarray*}
So
\begin{eqnarray*}
&&\int_{\mathbb C^n} \left \vert\langle C_\varphi k_z, k_w\rangle \right \vert^pe^{\left (\frac p2-1 \right )|z-w|^2}dV(w)\\
&=&e^{-p\frac {|z|^2}2}\int_{\mathbb C^n} e^{-p\frac {|w|^2}2} e^{p\{\Re e\hskip 1truemm \langle w, A^\star z\rangle+\langle B, z\rangle \}}e^{\left (\frac p2-1 \right )\{|z|^2+|w|^2-2\Re e \hskip 1truemm \langle w, z\rangle \}}dV(w)\\
&=&e^{-|z|^2+p\Re e\hskip 1truemm \langle B, z\rangle}\int_{\mathbb C^n} e^{-|w|^2}e^{\Re e\hskip 1truemm \langle w, pA^\star z\rangle-(p-2)\langle w, z\rangle}dV(w)\\
&=&e^{-|z|^2+p\Re e\hskip 1truemm \langle B, z\rangle}\int_{\mathbb C^n} e^{-\left \vert w-\frac p2 A^\star z+\left (\frac p2-1\right )z\right \vert^2}e^{\left \vert\frac p2 A^\star z-\left (\frac p2-1\right )z\right \vert^2}dV(w)\\
&=&e^{-|z|^2+p\Re e\hskip 1truemm \langle z, B\rangle+\frac {p^2}4|A^\star z|^2+\left (\frac p2-1\right )^2|z|^2-\Re e\hskip 1truemm \langle z, p\left (\frac p2-1\right )A^\star z \rangle}\int_{\mathbb C^n} e^{-|\zeta|^2}dV(\zeta),
\end{eqnarray*}
where we applied the change of variable $\zeta=w-\frac p2 A^\star z+\left (\frac p2-1\right )z.$ For $C:=\int_{\mathbb C^n} e^{-|\zeta|^2}dV(\zeta)<\infty,$ we obtain:
\begin{equation}\label{eq71}
\int_{\mathbb C^n} \left \vert\langle C_\varphi k_z, k_w\rangle \right \vert^pe^{\left (\frac p2-1 \right )|z-w|^2}dV(w)=Ce^{p\left (\frac p4-1\right )|z|^2+\Re e\hskip 1truemm \langle z, pB-p\left (\frac p2-1\right )A^\star z\rangle+\frac {p^2}4\left \vert A^\star z\right \vert^2}.
\end{equation}

Our first result is the following.
\begin{prop}\label{pro76}
We  suppose that $\left \Vert A\right \Vert <1.$ The composition operator $C_\varphi$ belongs to $\mathcal L_p$ if $2<p<\frac 4{1+\left \Vert A\right \Vert}.$ In particular, $C_\varphi$ is sufficiently localized.
\end{prop}

\begin{proof}
According to \eqref{eq71}, we must show that $\sup \limits_{z\in \mathbb C^n} e^{\frac p4g(z, p)}<\infty,$ where 
$$g(z, p)=\left (p-4\right )|z|^2+\Re e\hskip 1truemm \langle z, 4B-(2p-4)A^\star z\rangle+p\left \vert A^\star z\right \vert^2.$$
It suffices to prove that $\lim \limits_{|z|\rightarrow \infty} g(z, p)=-\infty.$ Indeed, in this case, for every $\alpha>0,$ there exists $\beta=\beta (\alpha)$ such that for every $z\in \mathbb C^n,$ the following implication holds.
$$|z|>\beta \Rightarrow g(z, p)<-\alpha.$$
We fix $\alpha$ and $\beta.$ Since the closed ball $\bar{B}(0, \beta)$ is a compact set of $\mathbb C^n$ and $g$ is a continuous function in $\mathbb C^n,$ there exists a real number $M$ such that the following implication also holds.
$$|z|\leq \beta \Rightarrow g(z, p)\leq M.$$
We then obtain that $\sup \limits_{z\in \mathbb C^n} e^{\frac p4g(z, p)}<e^{\frac p4\max \{-\alpha,M\}}<\infty.$

Let us next show that $\lim \limits_{|z|\rightarrow \infty} g(z, p)=-\infty.$ We have 
\begin{eqnarray*}
g(z, p)&\leq & (p-4)|z|^2+4|B||z|+(2p-4)\left \Vert A \right \Vert |z|^2+p\left \Vert A\right \Vert^2 |z|^2\\
&=&|z|^2 \left \{p-4+\left (2p-4\right )\left \Vert A \right \Vert +p\left \Vert A\right \Vert^2+\frac {4|B|}{|z|}\right \}.
\end{eqnarray*}
It is easy to check that $p-4+\left (2p-4\right )\left \Vert A \right \Vert +p\left \Vert A\right \Vert^2 <0$ whenever $2<p<\frac 4{1+\left \Vert A\right \Vert}.$ The conclusion follows.
\end{proof}

We next suppose that $\left \vert A\right \vert=1.$ We record the following theorem \cite[Theorem 3]{CMS03} (cf. originally \cite{HJ90}).

\begin{thm}\label{thm77}
If $A$ is an $n\times n$ matrix of rank $k,$ then $A$ can be written as $A=V\Sigma W,$ where $V, W$ are $n\times n$ unitary matrices, and $\Sigma$ is a diagonal matrix whose diagonal entries $\{\sigma_i\}_{i=1}^n$ are such that $\sigma_1\geq \sigma_2\geq \cdots \geq \sigma_k>\sigma_{k+1}=\cdots=\sigma_n=0.$ The $\sigma_i$ are the non negative square roots of the eigenvalues of $AA^\star;$ if we require that they are listed in decreasing order, then $\Sigma$ is uniquely determined from $A.$ 

If $\left \Vert A\right \Vert <1,$ then the $\sigma_i$ will all be less than or equal to 1, and at least one will be equal to 1 if $\left \Vert A\right \Vert=1.$
\end{thm}

This decomposition of the matrix $A$ is called {\textit {singular value decomposition.}} 

\subsection{The case where $A$ is a unitary $n\times n$ matrix}\label{ssec75}
Let $W$ be a unitary $n\times n$ matrix in $\mathbb C^n$ and let $B$ be an $n\times 1$ matrix in $\mathbb C^n.$ For all $\zeta \in \mathbb C^n,$ we have $|W\zeta|=|\zeta|.$ Then if the composition operator $C_\varphi, \varphi (z)=Wz+B$ is bounded on $H^2 (\mathbb C^n, d\mu),$ according to Theorem \ref{thm72}
, we have $\langle W\zeta, B\rangle=0$ for all $\zeta \in \mathbb C^n:$ since $W$ is onto, we conclude that $B=0.$

We denote by $\mathbb S$ the unit sphere in $\mathbb C^n.$ We first prove the following proposition.
\begin{prop}\label{pro78}
Let $W$ be a unitary $n\times n$ matrix in $\mathbb C^n$ and let $z'$ be a point of $\mathbb S.$ The following two assertions are equivalent.
\begin{enumerate}
\item
$\Re e\hskip 1truemm \langle Wz', z'\rangle=1;$
\item
$z'=Wz'.$
\end{enumerate}
\end{prop}

\begin{proof}
The implication $(2)\Rightarrow (1)$ is clear. Conversely, suppose (1), that is:
$$\Re e\hskip 1truemm \langle Wz', z'\rangle=1=|z'|^2=\frac {|z'|^2+|Wz'|^2}2.$$
Then
$$0=|Wz'|^2-2\Re e\hskip 1truemm \langle Wz', z'\rangle+|z'|^2=|Wz'-z'|^2.$$
This implies that $Wz'=z'.$ 
\end{proof}

We next deduce the following corollary.

\begin{cor}\label{cor79}
The only unitary $n\times n$ matrix $W$ in $\mathbb C^n$ whose composition operator $C_\varphi, \varphi (z)=Wz$ is weakly localized is the identity matrix. 
\end{cor}

\begin{proof}
Denote by $I$ the identity matrix in $\mathbb C^n.$ Then for $\varphi (z)=Iz=z,$ it is well-known that $C_\varphi$ is strongly localized. Next suppose that $W\neq I,$ that is, there exists a $z'\in \mathbb S$ such that $z'\neq Wz'.$

We first suppose that for all $z'\in \mathbb S,$ we have $z'\neq Wz'.$ According to the previous lemma, for all $z'\in \mathbb S,$ we have $\Re e\hskip 1truemm \langle Wz', z'\rangle<1=|z'|^2;$ since $\mathbb S$ is a compact set in $\mathbb C^n,$ there is a constant $C\in (0, 1)$ such that for all $z'\in \mathbb S,$ we have $\Re e\hskip 1truemm \langle Wz', z'\rangle \leq C.$ We claim that the Berezin transform $\widetilde{C_\varphi}$ of $C_\varphi$ vanishes at infinity; since $C_\varphi$ is not compact on the Fock space $H^2 (\mathbb C^n, d\mu),$ it follows from Theorem \ref{thm11}
  that $C_\varphi$ is not weakly localized. Indeed, for the claim:
$$\widetilde{C_\varphi} (z)=e^{-|z|^2+\Re e\hskip 1truemm \langle Wz, z\rangle};$$
for $z\neq 0,$ if we write $z=|z|z',$ we get:
\begin{eqnarray*}
|z|^2-\Re e\hskip 1truemm \langle Wz, z\rangle
&=&|z|^2 \left (1-  \Re e\hskip 1truemm \langle Wz, z\rangle \right )\\
&\geq& |z|^2 (1-C).
\end{eqnarray*}
This implies that $\widetilde{C_\varphi} (z)\rightarrow 0\quad (|z|\rightarrow \infty).$

We now suppose that there are two points $z'_1$ and $z'_2$ on $\mathbb S$ such that $Wz'_1=z'_1$ and $Wz'_1\neq z'_2.$ We recall that each unitary $n\times n$ matrix $W$ can be diagonalized in the form $W=VDV^\star,$ where $V$ is a unitary $n\times n$ matrix and $D$ is a diagonal matrix whose entries $\{d_j\}_{j=1}^n$ are unimodular. In our case, since $z'_1$ is an eigenvector of $W$ associated to the eigenvalue 1, we can take the successive entries of $D$ equal to $1, d_2,\cdots, d_n,$ with $d_2\neq 1.$ Then, in view of the equality: 
$$\langle C_\varphi k_z, k_w\rangle=e^{-\frac {|z|^2+|w|}2}e^{\langle Ww, z\rangle}=e^{-\frac {|V^\star z|^2+|V^\star w|}2}e^{\langle DV^\star w, V^\star z\rangle},$$ 
when we apply the change of variables $\xi=V^\star w$ and $\zeta=V^\star z,$ we get:
$$
\int_{|z-w|>r} \left \vert \langle C_\varphi k_z, k_w\rangle \right \vert dV(w)=\int_{|\zeta-\xi|>r} e^{-\frac {|\zeta|^2+|\xi|^2}2}e^{\Re e\hskip 1truemm \langle D\xi, \zeta\rangle}dV(\xi).
$$
We use the implication:
$$\{\zeta_1, \xi_1, \zeta_3, \xi_3,\cdots, \zeta_n, \xi_n \hskip 2truemm {\rm {arbitrary}},\hskip 2truemm |\zeta_2-\xi_2|>r\}\quad \Rightarrow \quad |\zeta-\xi|>r.$$
Then:
\begin{multline*}
\int_{|z-w|>r} \left \vert \langle C_\varphi k_z, k_w\rangle \right \vert dV(w)\geq \\
\int_{\mathbb C} e^{-\frac {\zeta_1-\xi_1|^2}2}d\lambda (\xi_1)\times  \Pi_{j=3}^n \int_{\mathbb C} e^{-\frac {|d_j \zeta_j-\xi_j|^2}2}d\lambda (\xi_j)\times \int_{|\zeta_2-\xi_2|>r} e^{-\frac {|d_2 \zeta_2-\xi_2|^2}2}d\lambda (z_2),
\end{multline*}
where $d\lambda$ denotes the Lebesgue area measure in $\mathbb C.$ If we write $C:=\int_{\mathbb C} e^{-\frac {|d_j\zeta_j-\zeta_j|^2}2}d\lambda (\xi_j)<\infty,$ we obtain:
\begin{equation*}
\sup \limits_{z\in \mathbb C^n} \int_{|z-w|>r} \left \vert \langle C_\varphi k_z, k_w\rangle \right \vert dV(w)\geq C^{n-1} \sup \limits_{\zeta_2\in \mathbb C}\int_{|\zeta_2-\xi_2|>r} e^{-\frac {|d_2 \zeta_2-\xi_2|^2}2}d\lambda (\xi_2).
\end{equation*}
If $C_\varphi$ were weakly localized, this would contradict assertion (2) of Proposition \ref{pro75}.
 The conclusion follows.

\begin{rem}\label{rem710}
In the above proof, we even proved that the only unitary $n\times n$ matrix $W$ in $\mathbb C^n$ whose composition operator $C_\varphi, \varphi (z)=Wz$ belongs to the Toeplitz algebra $\mathcal T^1$ is the identity matrix. 
\end{rem}
\end{proof}

\subsection{The case where $A=V\Sigma V^\star$}\label{ssec76}
In this subsection, we suppose that in its spectral value decomposition, $A=V\Sigma V^\star,$ that is $W=V^\star.$ We first prove the following lemma.

\begin{lem}\label{lem711}
Let $V$ and $U$ be a unitary $n\times n$ matrix in $\mathbb C^n$ and a unitary $(n-j)\times (n-j)$ unitary matrix in $\mathbb C^{n-j}$respectively. Let $B$ be a $n\times 1$ matrix in $\mathbb C^n$ and $D$ be a diagonal $(n-j)\times (n-j)$ matrix in $\mathbb C^{n-j}$ whose successive entries are $\sigma_{j+1}\geq \cdots \sigma_n$ with all $\sigma_j\in [0, 1), l\in \{j+1, \cdots, n\}.$ For $\Sigma_1=\left(
\begin{array}{r|l}
I_j&0\\ 
\hline
0&DU
\end{array}
\right)$ 
and $A=V\Sigma_1 V^\star,$ we suppose that the composition operator $C_\varphi, \varphi (z)=Az+B,$ is bounded on the Fock space $H^2 (\mathbb C^n, d\mu).$ If we denote by $B'_i, i=1,\cdots, n,$ the coordinates of the $n\times 1$ matrix $B'=V^\star B,$ then $B'_i=0$ for all $i\in \{1,\cdots, j\}.$
\end{lem}

\begin{proof} [Proof of the Lemma]
It suffices to show that $\langle B', \xi\rangle =0$ for all $\xi=\left (\xi_1,\cdots,\xi_j,0,\cdots,0\right )^T\in \mathbb C^j\times \{0_{\mathbb C^{n-j}}\}.$ According to $(1)$ in Theorem \ref{thm72},
 we have the implication $|Az|=|z|\Rightarrow \langle Az, B\rangle=0.$ We take $z=V\xi$ with $\xi=\left (\xi_1,\cdots,\xi_j,0,\cdots,0\right )^T.$ We get $Az=V\Sigma_1 V^\star z=V\Sigma_1 \xi=V\xi=z,$ since $\Sigma_1 \xi=\xi.$ Hence $|Az|=|z|.$ This implies that $\langle Az, B\rangle=0,$ or equivalently $\langle \Sigma_1 V^\star z, B'\rangle=0,$ or equivalently $\langle \Sigma_1 \xi, B'\rangle=\langle \xi, B'\rangle=0.$ This concludes the proof.
\end{proof}

We prove the following proposition.

\begin{prop}\label{pro712}
Let $A$ and $B$ be an $n\times n$ matrix and an $n\times 1$ matrix in $\mathbb C^n$ respectively, as described in $(1)$ of Theorem \ref{thm72}. 
 We suppose that in its spectral value decomposition, $A=V\Sigma V^\star,$ where $V$ is a unitary $n\times n$ matrix in $\mathbb C^n$ and the diagonal entries of $\Sigma$ are $\sigma_1=\cdots =\sigma_j=1, \sigma_{j+1}\geq \cdots \geq \sigma_n,$ with $\sigma_l \in [0, 1), l=k+1,\cdots,n.$ Then for $\varphi (z)=Az+B,$ the associated composition operator $C_\varphi$ is sufficiently localized. 
More precisely, the composition operator $C_\varphi$ belongs to $\mathcal L_p, 2<p<4.$
\end{prop}

\begin{proof}
From Lemma \ref{lem33}
 and the equality (\ref{eq71}),
  it follows that
$$\int_{\mathbb C^n} \left \vert U_zC_\varphi U_z \mathbbm 1\right \vert^p d\mu=e^{\frac p4g(z, p)},$$
where
$$g(z, p)=(p-4)|z|^2+\Re e\hskip 1truemm \langle z, 4B-(2p-4)A^\star z\rangle+p|A^\star z|^2.$$
We have $A^\star z=V\Sigma V^\star z$ and then $|A^\star z|=|\Sigma V^\star z|$ and $\Re e\hskip 1truemm \langle z, A^\star z \rangle=\Re e\hskip 1truemm \langle V^\star z, \Sigma V^\star z\rangle.$ Then:
$$g(z, p)=(p-4)|V^\star z|^2-(2p-4)\Re e\hskip 1truemm \langle V^\star z, \Sigma V^\star z\rangle+4\Re e\hskip 1truemm \langle V^\star z, B'\rangle+p|\Sigma V^\star z|^2,$$
where $B'=V^\star B.$ To simplify, we apply the change of variable $w=V^\star z;$ we obtain:

\begin{eqnarray*}
g(z, p) 
&=&(p-4)|w|^2-(2p-4)\Re e\hskip 1truemm \langle w, \Sigma w \rangle+4\Re e\hskip 1truemm \langle w, B'\rangle+p|\Sigma w|^2\\
&=&(p-4)\left (|w_1|^2+\cdots +|w_j|^2+|w_{j+1}|^2+\cdots +|w_n|^2\right )\\
&& -(2p-4)\left (|w_1|^2+\cdots +|w_j|^2+\sigma_{j+1}|w_{j+1}|^2+\cdots +\sigma_n|w_n|^2\right ) \\
&&  +4\Re e\hskip 1truemm \left (w_1\overline{B'_1}+\cdots w_j\overline{B'_j}+w_{j+1}\overline{B'_{j+1}}+w_n \overline{B'_n}  \right )\\
&&+p\left (|w_1|^2+\cdots +|w_j|^2+\sigma_{j+1}^2|w_{j+1}|^2+\cdots +\sigma_n^2|w_n|^2\right )\\
&=&\sum_{l=j+1}^n |w_l|^2\left \{(p-4)-(2p-4)\sigma_j+p\sigma_l^2  \right \}+\\
&& 4\Re e\hskip 1truemm \left (w_1\overline{B'_1}+\cdots +w_j\overline{B'_j}+w_{j+1}\overline{B'_{j+1}}+w_n \overline{B'_n}  \right ).
\end{eqnarray*}
From Lemma \ref{lem711}
 with $U=I_{n-j}$, we recall that $B'_i=0$ for all $i\in \{1,\cdots, j\}.$ So 
$$g(z, p)=\sum_{l=j+1}^n \left \{|w_l|^2\left ((p-4)-(2p-4)\sigma_l+p\sigma_l^2\right )+4\Re e\hskip 1truemm \left (w_l\overline{B'_l}\right )\right \}.$$
It is easy to check that 
\begin{equation}\label{eq72}
(p-4)-(2p-4)\sigma_l+p\sigma_l^2<0 
\end{equation}
for all $l\in \{j+1,\cdots, n\}$ and $p\in (2, 4).$ Indeed, consider this expression as a trinomial in $\sigma_j$ and use the assumptions $\sigma_j \in [0, 1), j=k+1,\cdots,n$ and $2<p<4.$ We conclude that $\lim \limits_{|z|\rightarrow \infty} g(z, p)=-\infty$ for all $2<p<4.$ 
The conclusion follows that the composition operator $C_\varphi$ belongs to $\mathcal L_p, 2<p<4.$ 
\end{proof} 

\subsection{The case where $A=V\Sigma W$}\label{ssec77}
According to Theorem \ref{thm77},
 we suppose that $A=V\Sigma W,$ where $V$ and $W$ are unitary $n\times n$ matrices in $\mathbb C^n$ and 
$\Sigma=\left (
\begin{array}{r|l}
I_j&0\\ 
\hline
0&D
\end{array}
\right )
,
$
where $I_j$ is the $j\times j$ identity matrix in $\mathbb C^j$ and $D$ is the $(n-j)\times (n-j)$ diagonal matrix whose successive diagonal entries $\sigma_{j+1}, \cdots, \sigma_n$ all belong to [0, 1) and are such that  $\sigma_{j+1}\geq \cdots \sigma_n.$ 
We next prove the following theorem.
\begin{thm}\label{thm713}
We suppose that $WV$ is not of the form {\rm { $\left (
\begin{array}{r|l}
I_j&0\\ 
\hline
0&U
\end{array}
\right )
$
}}
for some unitary $(n-j)\times (n-j)$ matrix $U$ in $\mathbb C^{n-j}.$ Then $C_\varphi, \varphi (z)=V\Sigma W z+B,$ is not sufficiently localized.
\end{thm}

\begin{proof}
It is easy to check that the following two assertions are equivalent.
\begin{enumerate}
\item
$WV \left (
\begin{array}{clcr}
\zeta_1\\
\vdots\\
\zeta_j\\
0\\
\vdots\\
0
\end{array}
\right )
=\left (
\begin{array}{clcr}
\zeta_1\\
\vdots\\
\zeta_j\\
0\\
\vdots\\
0
\end{array}
\right )
$
for all $(\zeta_1,\cdots, \zeta_j)\in \mathbb C^j\setminus \{0_{\mathbb C^{n-j}}\};$
\item
$WV=\left (
\begin{array}{r|l}
I_j&0\\ 
\hline
0&U
\end{array}
\right )
$
for some unitary $(n-j)\times (n-j)$ matrix $U$ in $\mathbb C^{n-j}.$
\end{enumerate}

We now suppose that $WV$ is not of the form 
$WV=\left (
\begin{array}{r|l}
I_j&0\\ 
\hline
0&U
\end{array}
\right )
$ Then there exists $(\zeta_1,\cdots, \zeta_j) \in \mathbb C^j \setminus \{0_{\mathbb C^j}\}$ such that $WV \left (
\begin{array}{clcr}
\zeta_1\\
\vdots\\
\zeta_j\\
0\\
\vdots\\
0
\end{array}
\right )
\neq \left (
\begin{array}{clcr}
\zeta_1\\
\vdots\\
\zeta_j\\
0\\
\vdots\\
0
\end{array}
\right ).$ We write $\zeta=\left (
\begin{array}{clcr}
\zeta_1\\
\vdots\\
\zeta_j\\
0\\
\vdots\\
0
\end{array}
\right ).$ According to Proposition \ref{pro78},
 we get:
$$\Re e\hskip 1truemm \langle WV\zeta, \zeta\rangle<|\zeta|^2.$$ 
It follows from (\ref{eq71}) 
that
$$\int_{\mathbb C^n} \left \vert U_z C_\varphi U_z \mathbbm 1\right \vert^p d\mu=Ce^{\frac p4 g(z, p)},$$
where
$$g(z, p)=\left (p-4\right )|z|^2+\Re e\hskip 1truemm \langle z, 4B-\left (2p-4\right )W^\star \Sigma V^\star z\rangle+p\left \vert W^\star \Sigma V^\star z\right \vert^2.$$
We take $z=V\zeta;$ then $\Sigma V^\star z=\zeta$ and $W^\star \Sigma V^\star z=W^\star \zeta.$ So
\begin{eqnarray*}
g(z, p)&=&\left (p-4+p\right )|\zeta|^2-\left (2p-4\right )\Re e\hskip 1truemm \langle V\zeta, W^\star \zeta\rangle +4\Re e\hskip 1truemm \langle \zeta, B'\rangle\\
&=&\left (2p-4\right )\left \{|\zeta|^2-\Re e\hskip 1truemm \langle WV \zeta, \zeta\rangle \right \} +4\Re e\hskip 1truemm \langle \zeta, B'\rangle .
\end{eqnarray*}
Here, $B'=V^\star B.$ Since $|\zeta|^2-\Re e\hskip 1truemm \langle WV \zeta, \zeta\rangle >0,$ we obtain that
$$g(Rz, p)\rightarrow \infty \quad (R\rightarrow \infty).$$
This implies that $\sup \limits_{z\in \mathbb C^n} \int_{\mathbb C^2} \left \vert U_z C_\varphi U_z \mathbbm 1\right \vert^p d\mu=\infty$ for all $p>2:$ so $C_\varphi$ is not sufficiently localized.
\end{proof}

Under the hypotheses of Theorem \ref{thm713},
 one may ask the question whether $C_\varphi, \varphi (z)=V\Sigma Wz+B,$ is weakly localized. We answer this question in the negative for $n=2.$ For this case, we show that $C_\varphi$ is not even a member of the Toeplitz algebra $\mathcal T^1.$ One would expect that this result extends to higher dimension.
\begin{thm}\label{thm714}
We suppose that $n=2$ and $WV$ is not of the form {\rm {$\left (
\begin{array}{clcr}
1&0\\ 
0&e^{i\theta}
\end{array}
\right )
$
}}
for some real number $\theta.$ Then $C_\varphi, \varphi (z)=V\Sigma Wz+B,$ is not a member of the Toeplitz algebra $\mathcal T^1.$
\end{thm}

\begin{proof}
We recall that according to Theorem \ref{thm73},
 $C_\varphi$ is not a compact operator on $H^2 (\mathbb C^2, d\mu).$ We shall prove that its Berezin transform $\widetilde {C_\varphi}$ vanishes at infinity. It will then follow from Theorem \ref{thm11}
  that $C_\varphi$ is not a member of $\mathcal T^1.$ From Lemma \ref{lem74}, we must show that
\begin{equation*}
\lim \limits_{|z|\rightarrow \infty} \left \{|z|^2-\Re e \hskip 1truemm \left (\langle Az, z\rangle +\langle B, z\rangle\right )\right \}=\infty.
\end{equation*}
We use a contradiction argument. Otherwise, there would exist a $z\in \mathbb C^2\setminus \{0\}$ such that
$$|z|^2=\Re e \hskip 1truemm \langle Az, z\rangle=\Re e \hskip 1truemm  (\langle V\Sigma Wz, z\rangle).$$
If we write $z=V\zeta,$ This amounts to
\begin{equation}\label{eq73}
|\zeta|^2=\Re e \hskip 1truemm \left(\langle \Sigma WV\zeta, \zeta\rangle\right)
\end{equation}
for some $\zeta\in \mathbb C^2\setminus \{0\}.$ Since $|\zeta|^2\geq \Re e \hskip 1truemm \left(\langle \Sigma WV\zeta, \zeta\rangle\right)$ for all $\zeta \in \mathbb C^2,$ such a solution $\zeta$ of \eqref{eq73} is an extremum of the quadratic form $q(\zeta):=|\zeta|^2-\Re e \hskip 1truemm \left(\langle \Sigma WV\zeta, \zeta\rangle\right).$ We write $\zeta =(\zeta_1, \zeta_2), \hskip 1truemm \Sigma=\left (
\begin{array}{clcr}
1&0\\
0&\sigma
\end{array}
\right )
$
with $\sigma \in [0, 1)$ and $WV=\left (
\begin{array}{clcr}
a_1&a_2\\
b_1&b_2
\end{array}
\right ).
$ The quadratic form is equal to
$$|\zeta|^2=\Re e \hskip 1truemm \left \{a_1|\zeta_1|^2+a_2\zeta_2\overline{\zeta_1}+\sigma\left (b_1\zeta_1\overline{\zeta_2}+b_2|\zeta_2|^2 \right )\right \};$$
for $\zeta_1=u_1+iv_1, \hskip 2truemm \zeta_2=u_2+iv_2,$ it takes the real form
\begin{multline*}
	q(u, v)=\left (u_1^2+v_1^2\right )\left (1- \Re e \hskip 1truemm a_1\right )+\left (u_2^2+v_2^2\right )\left (1-\sigma \Re e \hskip 1truemm b_2\right )-\\
	\Re e \hskip 1truemm \left \{a_2\left (u_2+iv_2\right )\left (u_1-iv_1\right )+\sigma b_1 \left (u_1+iv_1\right )\left (u_2-iv_2\right )\right \}.
\end{multline*}
The point $\zeta=(\zeta_1, \zeta_2)=\left (u_1+iv_1, u_2+iv_2\right )$ yields an extremum of $q(u, v)$ only if it is a stationary point, that is, a solution of the following system of four equations:
$$
\left \{
\begin{array}{cl}
\frac \partial {\partial u_j}q(u, v)&=0\\
\frac \partial {\partial v_j}q(u, v)&=0
\end{array}
\right.
\quad (j=1, 2).
$$ 
An easy calculation gives that a solution $\zeta=(\zeta_1, \zeta_2)$ of this system is such that 
$$
\left \{
\begin{array}{clcr}
\left (1-\Re e \hskip 1truemm a_1\right )\zeta_1&=\left (a_2+\sigma \overline{b_1}\right )\zeta_2\\
\left (1-\sigma\Re e \hskip 1truemm b_2\right )\zeta_2&=-\left (\overline{a_2}+\sigma b_1\right )\zeta_1
\end{array}
\right 
.
$$
The case where $1-\Re e \hskip 1truemm a_1=0$ corresponds to $a_1=1;$ since $WV$ is unitary, this implies that $a_2=0=b_1$ and $b_2=e^{i\theta}$ for some real number $\theta.$ So $WV=\left (
\begin{array}{clcr}
1&0\\
0&e^{i\theta}
\end{array}
\right ):
$ this case was excluded in the hypotheses of the theorem.\\
We next suppose that $1-\Re e \hskip 1truemm a_1 \neq 0.$ We deduce from the previous system that $$\zeta_2=-\frac {\left \vert a_2+\sigma \overline{b_1}\right \vert^2}{4\left (1-\Re e \hskip 1truemm a_1\right )\left (1-\sigma\Re e \hskip 1truemm b_2\right )}\zeta_2.$$ Since we exclude the zero solution, we get:
$$1=-\frac {\left \vert a_2+\sigma \overline{b_1}\right \vert^2}{4\left (1-\Re e \hskip 1truemm a_1\right )\left (1-\sigma\Re e \hskip 1truemm b_2\right )}.$$
We reached to a contradiction since the left-hand side is positive while the right-hand side is non-positive.

\end{proof}

We finally prove the following theorem.

\begin{thm}\label{thm715}
We suppose that $WV$ is of the form  $$ {\rm {\left(
\begin{array}{r|l}
I_j&0\\ 
\hline
0&U
\end{array}
\right)
}}$$
for some unitary $(n-j)\times (n-j)$ matrix $U$ in $\mathbb C^{n-j}$ different from $I_{n-j}.$ Then $C_\varphi, \varphi (z)=V\Sigma W z+B,$ is sufficiently localized.
\end{thm}

\begin{proof}
If $WV$ is of the form $\left (
\begin{array}{r|l}
I_j&0\\ 
\hline
0&U
\end{array}
\right )
$
for some unitary $(n-j)\times (n-j)$ matrix $U$ in $\mathbb C^{n-j}$ different from $I_{n-j},$ we have $W=\left (
\begin{array}{r|l}
I_j&0\\ 
\hline
0&U
\end{array}
\right )V^\star$ and $V\Sigma W=V\Sigma_1 V^\star,$ with 
$\Sigma_1=\left (
\begin{array}{r|l}
I_j&0\\ 
\hline
0&D
\end{array}
\right )\left (
\begin{array}{r|l}
I_j&0\\ 
\hline
0&U
\end{array}
\right )=\left (
\begin{array}{r|l}
I_j&0\\ 
\hline
0&DU
\end{array}
\right ).$

From Proposition \ref{pro33}
 and the equality \eqref{eq71}
it follows that
$$\int_{\mathbb C^n} \left \vert U_z C_\varphi U_z \mathbbm 1\right \vert^p d\mu=Ce^{\frac p4 g(z, p)},$$
where
$$g(z, p)=\left (p-4\right )|z|^2-\left (2p-4\right )\Re e\hskip 1truemm \langle z, V\Sigma_1^\star V^\star z\rangle+4\Re e \hskip 1truemm \langle z, B\rangle+p\left \vert V \Sigma_1 V^\star z\right \vert^2.$$
We have $\left \vert V \Sigma_1^\star V^\star z\right \vert=\left \vert \Sigma_1^\star V^\star z\right \vert, \Re e\hskip 1truemm \langle z, V\Sigma_1^\star V^\star z\rangle=\Re e\hskip 1truemm \langle V^\star z, \Sigma_1^\star V^\star z\rangle$ and $\langle z, B\rangle=\langle V^\star z, B'\rangle.$ Then 
$$g(z, p)=\left (p-4\right )|V^\star z|^2-\left (2p-4\right )\Re e\hskip 1truemm \langle V^\star z, \Sigma_1^\star V^\star z\rangle+4\Re e 1truemm \langle V^\star z, B'\rangle+p\left \vert \Sigma_1 V^\star z\right \vert^2,$$
where $B'=V^\star B.$ To simplify, we apply the change of variable $w=V^\star z;$ we obtain:
\begin{eqnarray*}
g(z, p)&=&\left (p-4\right )|w|^2-\left (2p-4\right )\Re e\hskip 1truemm \langle w, \Sigma_1^\star w\rangle+4\Re e \,\langle w, B'\rangle+p\left \vert \Sigma_1^\star w\right \vert^2\\
&=&\left (p-4\right )\left (|w_1|^2+\cdots +|w_j|^2+|w_{j+1}|^2+\cdots +|w_n|^2\right )-(2p-4)\left (|w_1|^2+\cdots +|w_j|^2+ \right.\\
&& \left.\Re e \,\,\langle w', U^\star Dw'\rangle \right ) +4\Re e \hskip 1truemm \left (w_1\overline {B'_1}+\cdots +w_j\overline {B'_j}+w_{j+1}\overline {B'_{j+1}}+\cdots +w_n\overline {B'_n}\right )\\
&& +p\left (|w_1|^2+\cdots +|w_j|^2+\left \vert U^\star Dw'\right \vert^2 \right )\\
&=&(p-4)|w'|^2-(2p-4)\Re e \hskip 1truemm \langle w', U^\star Dw'\rangle+p\left \vert Dw'\right \vert^2+4\Re e \hskip 1truemm \langle w', B''\rangle. 
\end{eqnarray*}
We applied Lemma 7.10 
 and the notations $w'=\left (w_{j+1},\cdots, w_n\right )^T$ and $B''=\left (B'_{j+1},\cdots, B'_n\right )^T.$

We deduce that
$$g(z, p)\leq (p-4)|w'|^2+(2p-4)|w'||Dw'|+p|Dw'|^2+4|w'||B''|.$$
It is easy to check that $|Dw'|\leq \left (\max_{l=j+1}^n \sigma_l\right )|w'|.$ We then obtain that
$$g(z, p)\leq |w'|^2\left (p-4+(2p-4)\max_{l=j+1}^n \sigma_l+p\max_{l=j+1}^n \sigma_l^2\right )+4|w'||B''|.$$
Since $\max_{l=j+1}^n \sigma_l <1,$ we get the inequality
$$p-4+(2p-4)\max_{l=j+1}^n \sigma_l+p\max_{l=j+1}^n \sigma_l^2<0$$
for all $p\in (2, 4)$ such that $0\leq \max_{l=j+1}^n \sigma_l <\frac {4-p}p.$ Such a $p$ exists because $\lim \limits_{p\stackrel{>}{\rightarrow} 2} \frac {4-p}p=1.$

This implies that $C_\varphi \in \mathcal L_p$ for all $p\in (2, 4)$ such that $0\leq \max_{l=j+1}^n \sigma_l <\frac {4-p}p.$ The conclusion follows.
\end{proof}

\section{The singular integral operators of convolution type introduced by K. Zhu}\label{sec8}
\subsection{Definitions and theorems}\label{ssec81}
For $\varphi\in H^2 (\mathbb C^n, d\mu),$ consider the integral operator $S_\varphi$ densely defined on $H^2 (\mathbb C^n, d\mu)$ by
\begin{equation}\label{eq81}
S_\varphi F(z)=\int_{\mathbb C^n} F(w)e^{\langle z, w\rangle}\varphi (z-\bar w)d\mu (w), \quad z\in \mathbb C^n, \quad F\in \mathcal S.
\end{equation}
It is easy to check that the operator $S_\varphi$ is admissible. Indeed, in the duality \eqref{eq21}, it holds: $\left (S_\varphi\right )^\star=S_\psi$ with $\psi (z)=\overline{\varphi (-\bar z)}.$

In 2015, K. Zhu \cite{Z15} proposed the following problem on the one-dimensional Fock space $H^2 (\mathbb C, d\mu)$ in the complex plane $\mathbb C.$ Characterize those functions $\varphi\in H^2 (\mathbb C, d\mu)$ such that the integral operator $S_\varphi$ in \eqref{eq81} is bounded on $H^2 (\mathbb C, d\mu).$

This problem was solved in 2020 for general $n\geq 1$ by Cao, Li, Shen, Wick and Yan \cite{CLSWY20}.
 For $\zeta =\left (\zeta_1,\cdots,\zeta_n \right )$ and $w =\left (w_1,\cdots,w_n \right )\in \mathbb C^n,$ the notations $\zeta^2$ and $\zeta\cdot w$ respectively stand for $\zeta^2=\zeta_1^2+\cdots+\zeta_n^2.$  Their result is the following.

\begin{thm}\label{thm81}
The integral operator $S_\varphi$ in \eqref{eq81} is bounded on $H^2 (\mathbb C, d\mu)$ if and only if there exists an $m\in L^\infty (\mathbb R^n)$ such that
\begin{equation}\label{eq82}
\varphi (z)=\int_{\mathbb R^n} m(x)e^{-\frac 12\left (x-iz\right )^2}dx, \quad z\in \mathbb C^n.
\end{equation}
Moreover, we have that
$$\left \Vert S_\varphi \right \Vert_{H^2 (\mathbb C^n, d\mu) \rightarrow H^2 (\mathbb C^n, d\mu)}=2^n\left \Vert m\right \Vert_{L^\infty (\mathbb R^n)}.$$
\end{thm}

The proof uses the Bargmann transform $B$ defined on $L^2 (\mathbb R^n)$ as follows
$$Bf (z)=\left (\frac 2\pi\right )^{\frac n4} e^{\frac {z^2}2}\int_{\mathbb R^n} f(x)e^{-(x-z)^2}dx, z\in \mathbb C^n.$$
The Bargmann transform $B$ is a unitary operator from $L^2 (\mathbb R^n)$ to $H^2 (\mathbb C, d\mu);$ it is one-to-one, onto, and isometric in the sense that
$$\int_{\mathbb R^n} |f(x)|^2dx=\int_{\mathbb C^n} |Bf (z)|^2d\mu (z).$$
It appears that if the integral operator $S_\varphi$ is bounded on $H^2 (\mathbb C, d\mu),$ then the operator $T:=B^{-1}S_\varphi B$ is bounded on $L^2 (\mathbb R^n)$ and commutes with translations. The proof of Theorem  \ref{thm81}
next relies on the following elementary fact from Harmonic Analysis characterizing the translation invariant operators that are bounded on $L^2 (\mathbb R^n)$  (\cite[Chapter II, Proposition 2]{S70}, \cite[Theorem 2.5]{G08}).

\begin{prop}\label{pro82}
Let $T$ be a bounded linear transformation mapping $L^2 (\mathbb R^n)$ into itself. The following two statements are equivalent.
\begin{enumerate}
\item
$T$ commutes with translations;
\item
There exists a bounded measurable function $m $ {\rm {(a "multiplier")}} so that $\widehat{Tf} (y)=m(y)\widehat f(y)$ for all $f\in L^2 (\mathbb R^n).$
\end{enumerate}
In this case the norm of $T: L^2 (\mathbb R^n)\rightarrow L^2 (\mathbb R^n)$ is equal to $\left \Vert m\right \Vert_{L^\infty}.$
\end{prop}

We next state the following remark.

\begin{rem}\label{rem83}
Suppose that $S_\varphi$ is bounded on $H^2 (\mathbb C, d\mu).$ Then its adjoint $\left (S_\varphi\right )^\star$ is given by $\left (S_\varphi\right )^\star=S_{\widetilde {\varphi}}$ with $\widetilde {\varphi} (z)=\overline{\varphi (-\bar z)}, \hskip 2truemm z\in \mathbb C^n$ and its associated multiplier is $\widetilde{m} (x)=\overline{m(x)}, \hskip 2truemm x\in \mathbb R^n.$ 
\end{rem}

In their paper, Cao, Li, Shen, Wick and Yan \cite{CLSWY20}
 also showed that a bounded operator $S_\varphi$ is compact if and only if $\varphi \equiv 0.$ We suppose that $S_\varphi$  is bounded on $H^2 (\mathbb C, d\mu)$ and we ask the following question. Is the operator $S_\varphi$ weakly localized? Are there weakly localized operators $S_\varphi$ which are not XZ-sufficiently localized? 

In view of Theorem \ref{thm11} 
and Theorem \ref{thm12}, we first prove that for each symbol $\varphi$ satisfying \eqref{eq82} and which does not vanish identically, the Berezin transform $\widetilde{S_\varphi}$ of the associated operator $S_\varphi$ does not vanish at infinity. Indeed, it follows from \eqref{eq81} and an application of the reproducing kernel formula that
\begin{equation}\label{eq83}
\left \vert\langle S_\varphi k_z, k_w\rangle\right \vert=e^{-\frac {|z-w|^2}2}\left \vert\varphi (w-\bar z)\right \vert
\end{equation}
and hence $\widetilde{S_\varphi} (z)=\langle S_\varphi k_z, k_z\rangle=\varphi \left (z-\bar z\right ).$ The identity $\varphi \left (z-\bar z\right )=0$ for all $z\in \mathbb C^n$ would imply that $\varphi \equiv 0$ in $\mathbb C^n$ according to the analytic continuation principle. Otherwise, there exists $z\in \mathbb C^n$ such that  $\varphi \left (z-\bar z\right )\neq 0.$ We keep $z-\bar z=2i\Im m\hskip 1truemm z$ fixed and we let $\Re e\hskip 1truemm z$ tend to infinity. Then $|z|\rightarrow \infty$ while $\varphi \left (z-\bar z\right )$ remains constant and different from zero.

We deduce the following corollary.

\begin{cor}\label{cor84}
We suppose that the symbol $\varphi$ satisfies \eqref{eq82} and does not vanish identically in $\mathbb C^n.$ The following two statements are equivalent.
\begin{enumerate}
\item
{\rm {i)}} $S_\varphi$ is weakly localized {\rm (resp. {ii)}} $S_\varphi$ is {\rm {XZ-}}sufficiently localized, {\rm iii)} $S_\varphi$ is sufficiently localized{\rm )};
\item
\begin{enumerate}
\item[i)]  We have
$$\lim \limits_{r\rightarrow \infty} \sup \limits_{z\in \mathbb C^n} \int_{|z-w|\geq r} e^{-\frac {|z-w|^2}2}\left \vert \varphi (w-\bar z) \right \vert dV(w)=0;$$ 
{\rm {(resp.}}
\item[ii)] there exist two positive constants $C$ and $\beta>2n$ such that for all $z,w\in \mathbb C^n,$ we have
$$e^{-\frac {|z-w|^2}2}\left \vert \varphi (w-\bar z) \right \vert \leq \frac C{\left (1+|w-z|\right )^\beta},$$
\item[iii)] we have $$\sup \limits_{z\in \mathbb C^n} \int_{\mathbb C^n} \left \vert \varphi (w-\bar z) \right \vert^p e^{-|w-z|^2}dV(w)<\infty$$
for some $p>2$ or equivalently, there exist two positive constants $C$ and $\epsilon$ such that for all $z,w\in \mathbb C^n,$ we have
$$e^{-\frac {|z-w|^2}2}\left \vert \varphi (w-\bar z) \right \vert \leq e^{-\epsilon |z-w|^2}{\rm ).}$$
\end{enumerate}
\end{enumerate}
\end{cor}

For $n=1,$ Bauer, Fulsche and Rodriguez \cite[Theorem 3.19, assertion 1]{BFR24}  
proved that a bounded operator $S_\varphi$ is a member of the Toeplitz algebra if and only if $m\in BUC \hskip 1truemm (\mathbb R).$ Using the same arguments, this result can be extended to upper dimensions to get the following theorem.

\begin{thm}\label{thm85}
Suppose that $n$ is a positive integer. The following two assertions are equivalent.
\begin{enumerate}
\item
The bounded operator $S_\varphi$ is a member of the Toeplitz algebra $\mathcal T^1;$
\item
$m\in BUC \hskip 1truemm (\mathbb R^n).$
\end{enumerate}
\end{thm}

Combining with the previous corollary, we obtain the following result, which may be of independent interest.

\begin{cor}\label{cor86}
We suppose that the bounded function $m$ in $\mathbb R^n$ satisfies the following estimate.
\begin{equation}\label{eq84}
	\lim \limits_{r\rightarrow \infty} \sup \limits_{y\in \mathbb R^n} \int_{|s|^2+|t|^2\geq r^2}e^{-\frac {|t|^2}2}\left \vert \int_{\mathbb R^n}   m\left (\xi-t-2\Im m\hskip 1truemm z\right )e^{-\frac{|\xi|^2}2}e^{i\xi\cdot s}d\xi\right \vert dsdt=0.
\end{equation}
Then $m\in BUC \hskip 1truemm (\mathbb R^n).$
\end{cor}

\begin{proof}
It suffices to show that the estimate \eqref{eq84} is equivalent to the fact that the bounded operator $S_\varphi$ associated to $m$ is weakly localized. This will follow from the following lemma.

\begin{lem}\label{lem87}
For all $z, w\in \mathbb C^n,$ we have the equality
$$e^{-\frac {|z-w|^2}2}\left \vert \varphi (w-\bar z) \right \vert=e^{-\frac {|t|^2}2}\left \vert \int_{\mathbb R^n}   m\left (\xi-t-2\Im m\hskip 1truemm z\right )e^{-\frac{|\xi|^2}2}e^{i\xi\cdot s}d\xi\right \vert,$$
where $s=\Re e\hskip 1truemm (w-z)$ and $t=\Im m\hskip 1truemm (w-z).$
\end{lem}

\begin{proof}[Proof of the lemma]
Applying successively the changes of variables $\zeta=w-z$ and $\zeta=s+it,$ and \eqref{eq82}, we have

\begin{eqnarray*}
&&e^{-\frac {|z-w|^2}2}\left \vert \varphi (w-\bar z) \right \vert\\
&=&
e^{-\frac {|\zeta|^2}2}\left \vert \varphi (\zeta+2i\Im m\hskip 1truemm z) \right \vert\\
&=&e^{-\frac {|s|^2+|t|^2}2}\left \vert \int_{\mathbb R^n} m(x)e^{-\frac 12\left (x-i(s+i(t+2\Im m\hskip 1truemm z)   \right )^2}dx\right \vert\\
&=&e^{-\frac {|s|^2+|t|^2}2}\left \vert \int_{\mathbb R^n} m(x)e^{-\frac 12\left (|x|^2-2ix\cdot (s+i(t+2\Im m\hskip 1truemm z-(s+i(t+2\Im m\hskip 1truemm z))^2 \right )}dx\right \vert\\
&=&e^{-\frac {|s|^2+|t|^2}2}e^{\frac 12\Re e\hskip 1truemm (s+i(t+2\Im m\hskip 1truemm z)^2}\left \vert \int_{\mathbb R^n} m(x)e^{-\frac 12|x|^2-x\cdot (t+2\Im m\hskip 1truemm z)}e^{ix\cdot s}dx\right \vert\\
&=&e^{-\frac {|s|^2+|t|^2}2}e^{\frac 12(|s|^2-|t+2\Im m\hskip 1truemm z|^2)}\left \vert \int_{\mathbb R^n} m(x)e^{-\frac 12\left (\left \vert x+(t+2\Im m\hskip 1truemm z)\right \vert^2+\left \vert \frac 12(t+2\Im m\hskip 1truemm z)\right \vert^2\right )}e^{ix\cdot s}dx\right \vert\\
&=&e^{-\frac {|t|^2}2}\left \vert \int_{\mathbb R^n} m(x)e^{-\frac 12\left \vert x+t+2\Im m\hskip 1truemm z\right \vert^2}e^{ix\cdot s}dx\right \vert\\
&=&e^{-\frac {|t|^2}2}\left \vert \int_{\mathbb R^n} m(\xi-t-2\Im m\hskip 1truemm z)e^{-\frac 12\left \vert \xi\right \vert^2}e^{i\xi\cdot s}d\xi\right \vert.
\end{eqnarray*}
\end{proof}
\end{proof}
The next question is whether the converse implication is true, i.e. if $m\in BUC \hskip 1truemm (\mathbb R^n),$ is the associated operator $S_\varphi$  weakly localized? We shall prove in the last subsection below that this is true in the particular case where $m$ is the Fourier transform of an integrable function $g$ in $\mathbb R^n.$  In this case, $m$ belongs to the space $\mathcal C_0 (\mathbb R^n)$ of continuous functions in $\mathbb R^n$ which vanish at infinity.

We next prove the following theorem.

\begin{thm}\label{thm88}
Every bounded operator $S_\varphi$ is a member of $\mathcal L_2.$
\end{thm}

\begin{proof}
We suppose that $\varphi$ satisfies \eqref{eq82} and does not vanish identically. According to Corollary \ref{cor84},
 we must show that
\begin{equation}\label{eq85}
\sup \limits_{z\in \mathbb C^n} \int_{\mathbb C^n} \left \vert \varphi (w-\bar z) \right \vert^2 e^{-|w-z|^2}dV(w)<\infty.
\end{equation}
For $\zeta=s+it,$ we have
\begin{eqnarray*}
\varphi (\zeta)
&=&\int_{\mathbb R^n} m(x)e^{-\frac 12\left (x-i(s+it) \right )^2}dx\\
&=&e^{\frac 12\left (|s|^2-|t|^2+2is\cdot t \right )}\int_{\mathbb R^n} m(x)e^{-\frac {|x|^2}2-x\cdot t}
e^{ix\cdot s}dx
\end{eqnarray*}
and then
$$\left \vert \varphi (s+it)\right \vert=e^{\frac 12\left (|s|^2-|t|^2\right )}\left \vert \int_{\mathbb R^n} m(x)e^{-\frac {|x|^2}2-x\cdot t}
e^{ix\cdot s}dx\right \vert.$$
Applying the change of variable $w-\bar z=\zeta,$ we get $w-z=\zeta-2i\Im  m\hskip 1truemm z$ and 
\begin{eqnarray*}
\int_{\mathbb C^n} \left \vert \varphi (w-\bar z) \right \vert^2 e^{-|w-z|^2}dV(w)
&=&\int_{\mathbb C^n} \left \vert \varphi (\zeta) \right \vert^2 e^{-|\zeta-2i\Im m \hskip 1truemm z|^2}dV(\zeta)\\
&=&\int_{\mathbb C^n} \left \vert \varphi (s+it) \right \vert^2 e^{-|s|^2-|t-2\Im m \hskip 1truemm z|^2}dsdt\\
&=&\int_{\mathbb C^n} e^{|s|^2-|t|^2}\left \vert \int_{\mathbb R^n} m(x)e^{-\frac {|x|^2}2-x\cdot t}
e^{ix\cdot s}dx\right \vert^2 e^{-|s|^2-|t-2\Im m \hskip 1truemm z|^2}dsdt\\
&=&\int_{\mathbb R^n} I(t, z)e^{-|t|^2-|t-2\Im m \hskip 1truemm z|^2}dt,
\end{eqnarray*}
where $I(t, z):=\int_{\mathbb R^n} \left \vert \int_{\mathbb R^n}  m(x)e^{-\frac{|x|^2}2-x\cdot t}
e^{ix\cdot s}dx\right \vert^2ds.$
By the Plancherel formula, we obtain:
\begin{eqnarray*}
I(t, z)
&=&C_n\int_{\mathbb R^n}  |m(x)|^2e^{-(|x|^2+2x\cdot t)}dx\\
&\leq& C_n\left \Vert m\right \Vert_{L^\infty}^2\int_{\mathbb R^n} e^{-(|x|^2+2x\cdot t)}dx\\
&=&C_n\left \Vert m\right \Vert_{L^\infty}^2\int_{\mathbb R^n} e^{-(|x+t|^2-|t|^2)}dx\\
&=&C'_n\left \Vert m\right \Vert_{L^\infty}^2 e^{|t|^2}.
\end{eqnarray*}
We conclude that
\begin{eqnarray*}
\int_{\mathbb C^n} \left \vert \varphi (w-\bar z) \right \vert^2 e^{-|w-z|^2}dV(w)
&=&C'_n\left \Vert m\right \Vert_{L^\infty}^2\int_{\mathbb R^n} e^{-|t-2\Im m \hskip 1truemm z|^2}dt\\
&=&C'_n\left \Vert m\right \Vert_{L^\infty}^2\int_{\mathbb R^n} e^{-|\tau|^2}d\tau.
\end{eqnarray*}
This implies the estimate \eqref{eq85}. 
\end{proof}

Combining with Theorem \ref{thm85},
 we obtain the following corollary.

\begin{cor}\label{cor89}
The inclusion $\mathcal T^1 \subset \mathcal L_2$ is strict.
\end{cor}

\begin{proof}
Suppose that $m\in L^\infty\setminus BUC (\mathbb R^n).$ Then $S_\varphi \in \mathcal L_2$ and $S_\varphi \notin \mathcal T^1$ by Theorem \ref{thm88}
and Theorem \ref{thm85}
respectively.
\end{proof}

\subsection{Examples}\label{ssec82}
We refer to \cite[Section 4]{CLSWY20}
. Cf. also \cite{Z15} and \cite{Z19} 
. In both examples, we have $n=1.$
\subsubsection*{Example 1}
The basic bounded linear operator mapping $L^2 (\mathbb R^n)$ into itself which commutes with translations is given for n=1 by the Hilbert transform defined as
$$Hf (x):=\frac 1\pi\int_{\mathbb R} \frac {f(y)}{x-y}dy,$$
where the improper integral is taken in the sense of "principal value". Note that $\widehat{Hf} (x)=m(x)\widehat f(x)$ with $m(x)=-i\hskip 1truemm sgn x.$ Obviously, this multiplier $m$ is bounded, but it is not continuous. So it follows from Theorem \ref{thm85}
  that the associated operator $S_\varphi=BHB^{-1}$ is not a member of the Toeplitz algebra $\mathcal T=\mathcal T (\mathbb C).$ It is proved in  \cite{CLSWY20}
    that in this example, $\varphi (z)=\frac 2{\sqrt \pi} A\left (\frac z{\sqrt 2} \right ),$ where
$$A(z):=\int_0^z e^{u^2}du, \quad z\in \mathbb C$$
which is the antiderivative of $e^{u^2}$ satisfying $A(0)=0.$

\subsubsection*{Example 2}
Let $m(x)=e^{-2iax}$ where $a$ is a real number. This function belongs to $BUC (\mathbb R).$ So the associated operator $S_\varphi$
is a member of the Toeplitz algebra $\mathcal T=\mathcal T (\mathbb C).$ More precisely, we have $\varphi (z)=ce^{-\frac {a^2}2}e^{az}$ where $c$ is a constant independent of $a$ and up to a multiplicative constant, $S_\varphi$
is the unitary operator $V_a$ defined in Subsection \ref{ssec41}
 It was proved there that this operator is strongly localized.

\section{The proof of the main result}\label{sec9}
\subsection{A further class of examples of singular operators of convolution type. The case where $m=\widehat{g}$ with $g\in L^1 (\mathbb R^n)$}\label{ssec91}
We suppose that $m=\widehat{g}$ with $g\in L^1 (\mathbb R^n).$ We first give the corresponding expression of $\varphi.$ 

\begin{lem}\label{lem91}
For $m=\widehat{g}$ with $g\in L^1 (\mathbb R^n),$ we have:
\begin{equation}\label{eq91}
\varphi (z)=\int_{\mathbb R^n} g(\sigma)e^{-\frac {\sigma^2}2}e^{-\sigma\cdot z}d\sigma.
\end{equation}
\end{lem}

\begin{proof}
For $z=it,\hskip 2truemm t\in \mathbb R^n,$ it follows from \eqref{eq82} that
$$\varphi (it)=\int_{\mathbb R^n} \widehat{g} (x)e^{-\frac 12(x+t)^2}dx.$$
Since the function $\xi \mapsto e^{-\frac {\xi^2}2}$ is its own Fourier transform, it follows from the inverse Fourier formula that
\begin{eqnarray*}
\varphi (it)&=&\int_{\mathbb R^n} g(\sigma)e^{-\frac 12(s-\sigma)^2}e^{i(s-\sigma)\cdot t}d\sigma|_{s=0}\\
&=&\int_{\mathbb R^n} g(\sigma)e^{-\frac {\sigma^2}2}e^{-i\sigma\cdot t}d\sigma.
\end{eqnarray*}

Applying the analytic continuation principle, we obtain that
$$\varphi (z)=\int_{\mathbb R^n} g(\sigma)e^{-\frac {\sigma^2}2}e^{-\sigma\cdot z}d\sigma.$$.
\end{proof}

We next prove the following theorem.

\begin{thm}\label{thm92}
Suppose that $m=\widehat{g}$ with $g\in L^1 (\mathbb R^n).$  Then its associated operator $S_\varphi$ is weakly localized. 
\end{thm}

\begin{proof}
It follows from \eqref{eq83} and the previous lemma that
\begin{eqnarray}
\left \vert \langle S_\varphi k_z, k_w\rangle \right \vert
&=&e^{-\frac {|z-w|^2}2}\left \vert \int_{\mathbb R^n} e^{\sigma\cdot \left (\bar{z}-w\right )}g(\sigma)e^{-\frac {\sigma^2}2}d\sigma\right \vert \label{Sphi}\\
&\leq& e^{-\frac {|z-w|^2}2} \int_{\mathbb R^n} e^{\sigma\cdot \Re e\hskip 1truemm \left (z-w\right )}|g(\sigma)|e^{-\frac {|\sigma|^2}2}d\sigma. \nonumber
\end{eqnarray}
Applying the change of variable $z-w=s+it,$ we obtain:
\begin{eqnarray*}
\int_{|z-w|\geq r} \left \vert \langle S_\varphi k_z, k_w\rangle \right \vert dV(w)
&\leq& \int_{s^2+t^2 \geq r^2} e^{-\frac {|s|^2+|t|^2}2} \left (\int_{\mathbb R^n} e^{\sigma \cdot s}|g(\sigma)|e^{-\frac {\sigma^2}2}d\sigma \right )dsdt\\
&=&\int_{s^2+t^2 \geq r^2} e^{-\frac {t^2}2} \left (\int_{\mathbb R^n} |g(\sigma)|e^{-\frac {(\sigma-s)^2}2}d\sigma \right )dsdt\\
&\leq& I_1 (r)+I_2 (r),
\end{eqnarray*}
where
$$I_1 (r):=\int_{|t| \geq \frac r2, \hskip 1truemm s\in \mathbb R^n} e^{-\frac {|t|^2}2} \left (\int_{\mathbb R^n} |g(\sigma)|e^{-\frac {(s-\sigma)^2}2}d\sigma\right ) dsdt$$
and 
$$I_2 (r):=\int_{|s| \geq \frac r2, \hskip 1truemm t\in \mathbb R^n} e^{-\frac {|t|^2}2} \left (\int_{\mathbb R^n} |g(\sigma)|e^{-\frac {(s-\sigma)^2}2}d\sigma\right ) dsdt.$$
We have used the inclusion:
\begin{multline*}
	\left \{(s, t)\in \mathbb R^n \times \mathbb R^n: |s|^2+|t|^2 \geq r^2\right \}\subset \\
	 \left \{(s, t)\in \mathbb R^n \times \mathbb R^n: |t| \geq \frac r2\right \}\cup \left \{(s, t)\in \mathbb R^n \times \mathbb R^n: |s| \geq \frac r2\right \}.
\end{multline*}
We write $C:=\int_{\mathbb R^n} e^{-\frac {\tau^2}2}d\tau <\infty$ and $C(r):=\int_{|t| \geq \frac r2} e^{-\frac {t^2}2}dt;$ we have $\lim \limits_{r\rightarrow \infty} C(r)=0.$ On the one hand, by the Fubini-Tonelli theorem, we get:
\begin{eqnarray*}
I_1 (r)&=&\left (\int_{|t| \geq \frac r2} e^{-\frac {t^2}2}dt \right )\int_{\mathbb R^n} \left (\int_{\mathbb R^n} |g(\sigma)|e^{-\frac {(s-\sigma)^2}2}d\sigma\right ) dsdt\\
&\leq &C(r)\int_{\mathbb R^n}  \left (\int_{\mathbb R^n} e^{-\frac {(s-\sigma)^2}2}ds\right )|g(\sigma)|d\sigma\\
&=&C(r)C\left \Vert g\right \Vert_{L^1 (\mathbb R^n)}\rightarrow 0
\end{eqnarray*}
as $r\rightarrow \infty.$ On the other hand, we have:
\begin{eqnarray*}
I_2 (r)&=&\left (\int_{\mathbb R^n} e^{-\frac {t^2}2}dt \right )\int_{|s| \geq \frac r2} \left \vert \int_{\mathbb R^n} g(\sigma)e^{-\frac {|s-\sigma|^2}2}d\sigma\right \vert ds\\
&=&C\int_{|s| \geq\frac r2} \left (\int_{\mathbb R^n} |g(\sigma)|e^{-\frac {|s-\sigma|^2}2}d\sigma\right ) ds.
\end{eqnarray*}
To get the estimate $\lim \limits_{r\rightarrow \infty} I_2(r)=0,$ it suffices to show that
$$\lim \limits_{r\rightarrow \infty} \int_{|s| \geq \frac r2} \left (\int_{\mathbb R^n} |g(\sigma)|e^{-\frac {(s-\sigma)^2}2}d\sigma \right )ds=0;$$
This conclusion will follow as soon as we prove that the function $s\mapsto \int_{\mathbb R^n} |g(\sigma)|e^{-\frac {|s-\sigma|^2}2}d\sigma$ is integrable in $\mathbb R^n.$ Indeed, by the Fubini-Tonelli theorem, we have:
\begin{eqnarray*}
\int_{\mathbb R^n} \left (\int_{\mathbb R^n} |g(\sigma)|e^{-\frac {(s-\sigma)^2}2}d\sigma \right )ds
&=&\int_{\mathbb R^n} \left (\int_{\mathbb R^n} e^{-\frac {(s-\sigma)^2}2}ds\right )|g(\sigma)|d\sigma\\
&=&C\left \Vert g\right \Vert_{L^1 (\mathbb R^n)}<\infty.
\end{eqnarray*}
The second estimate $\lim \limits_{r\rightarrow \infty} \sup \limits_{z\in \mathbb C^n} \int_{|z-w|>r} \left \vert \langle S_\varphi^\star k_z, k_w\rangle \right \vert dV(w)=0$ is proved similarly using Remark \ref{rem83} 
the adjoint operator $S_\varphi^\star$ is associated to the multiplier $\bar m=\widehat{\bar g}\hskip 2truemm (\bar g\in L^1 (\mathbb R^n).$ The operator $S_\varphi$ is weakly localized. 
\end{proof} 

We next show the existence of an operator $S_\varphi,$ associated to a multiplier $m=\widehat{g}, \hskip 2truemm g\in L^1 (\mathbb R^n),$ which is sufficiently localized. We recall from \eqref{Sphi} that
$$\left \vert \langle S_\varphi k_z, k_w\rangle\right \vert=e^{-\frac {|z-w|^2}2}\left \vert \int_{\mathbb R^n} e^{\sigma \cdot (\bar{z}-w)}g(\sigma)e^{-\frac {\sigma^2}2}d\sigma\right \vert.$$
We apply the change of variables $\zeta=z-w$ and $y=\Im m\hskip 1truemm z.$ Then:
\begin{eqnarray}\nonumber
\left \vert \langle S_\varphi k_z, k_w\rangle\right \vert
&=&e^{-\frac {|\zeta|^2}2}\left \vert \int_{\mathbb R^n} e^{\sigma \cdot (-2iy+\zeta)}g(\sigma)e^{-\frac {\sigma^2}2}d\sigma\right \vert \label{c1}\\
&=&e^{-\frac {|\zeta|^2}2}\left \vert \mathcal F \left (e^{\sigma \cdot \zeta}g(\sigma)e^{-\frac {\sigma^2}2}\right )(y)\right \vert. \label{c2}
\end{eqnarray}
Here, $\mathcal F$ stands for the Fourier transform.

We now prove the following lemma.

\begin{lem}\label{lem93}
Let $g\in L^1 (\mathbb R^n)$ be compactly supported in a closed Euclidean ball $\bar{B} (0, A),$ with $A<\frac 12.$ Then $S_\varphi \in \mathcal{SL}.$
\end{lem}

\begin{proof}
According to Corollary \ref{cor311}, it suffices to show that there exist two positive constants $C$ and $\epsilon$ such that for all $z, w\in \mathbb C^n,$ we have:
$$\left \vert \langle S_\varphi k_z, k_w\rangle\right \vert\leq Ce^{-\epsilon |\zeta|^2}.$$
It follows from \eqref{c1} that
\begin{eqnarray*}
\left \vert \langle S_\varphi k_z, k_w\rangle\right \vert
&\leq& e^{-\frac {|\zeta|^2}2}\int_{|\sigma|\leq A} e^{\sigma \cdot \Re e \hskip 1truemm \zeta}|g(\sigma)|e^{-\frac {\sigma^2}2}d\sigma\\
&\leq& e^{-\frac {|\zeta|^2}2}e^{A|\zeta|}\left \Vert g\right \Vert_{L^1 (\mathbb R^n)}\\
&\leq &Ce^{-|\zeta|^2\left (\frac 12-A\right )}.
\end{eqnarray*}
The required conclusion holds with $\epsilon=\frac 12-A.$
\end{proof}

We finally prove the following theorem, whose combination with Theorem \ref{thm92}
yields the second set inequality in our main result (Theorem \ref{thm13}). 

\begin{thm}\label{thm94}
Let $n=1.$ There exists a function $g\in L^1 (\mathbb R)$ such that the operator $S_\varphi$ associated to the multiplier $m=\widehat g$ is not XZ-sufficiently localized.
\end{thm}

\begin{proof}
We use a contradiction argument. We suppose that for every $g\in L^1 (\mathbb R),$ the operator $S_\varphi$ associated to the multiplier $m=\widehat g$ is XZ-sufficiently localized, i.e. there exist two positive constants $C$ and $\beta$ with $\beta>2,$ such that for all $z, w\in \mathbb C,$ we have
$$e^{-\frac {|z-w|^2}2}\left (1+|z-w|\right )^\beta \left \vert \int_{\mathbb R} e^{\sigma(\bar z-w)}g(\sigma)e^{-\frac {\sigma^2}2}d\sigma\right \vert \leq C.$$ 
We take $z=x\in \mathbb R$ and $w=0.$ For every $x\in \mathbb R,$ we obtain:
\begin{eqnarray} 
e^{-\frac {x^2}2}\left (1+|x|\right )^\beta \left \vert \int_{\mathbb R} e^{\sigma x}g(\sigma)e^{-\frac {\sigma^2}2}d\sigma \right \vert
&=&\left \vert \left (1+|x|\right )^\beta \int_{\mathbb R} g(\sigma)e^{-\frac {(\sigma-x)^2}2}d\sigma\right \vert \nonumber\\
&\leq& C. \label{der}
\end{eqnarray}
For every $x\in \mathbb R,$ we define the linear functional $T_x$ on $L^1 (\mathbb R)$ by 
$$T_x g=\left (1+|x|\right )^\beta \int_{\mathbb R} g(\sigma)e^{-\frac {(\sigma-x)^2}2}d\sigma.$$
Clearly,
$$\left \vert T_x g\right \vert\leq \left (1+|x|\right )^\beta \left \Vert g\right \Vert_{L^1 (\mathbb R)},
$$
since $\sup \limits_{x\in \mathbb R} e^{-\frac {x^2}2}=1.$ In other words, each linear functional $T_x, \hskip 2truemm x\in \mathbb R,$ is continuous on $L^1 (\mathbb R),$ i. e. each $T_x$ belongs to the topological dual $\left (L^1 (\mathbb R) \right )^*$ of $L^1 (\mathbb R).$

In view of \eqref{der}, for every $g\in L^1 (\mathbb R),$ we have
$$\sup \limits_{x\in \mathbb R} \left \vert T_x g\right \vert <\infty.$$
It then follows from the uniform boundedness principle  \cite[Theorem 2.2]{B11}
 that there exists a constant $c$ such that for all $x\in \mathbb R$ and $g\in L^1 (\mathbb R),$ we have
$$\left \vert T_x g\right \vert \leq c\left \Vert g\right \Vert_{L^1 (\mathbb R)}.$$
So by the Riesz representation theorem \cite[Theorem 4.14]{B11}
, for every $x\in \mathbb R,$ we get:
$$\sup \limits_{\sigma \in \mathbb R} \left (1+|x|\right )^\beta e^{-\frac {(\sigma-x)^2}2}\leq c.$$ 
We take $\sigma=x$ to obtain $\left (1+|x|\right )^\beta \leq c$ for every $x\in \mathbb R:$ this conclusion is of course false. We have reached to a contradiction. This finishes the proof of the theorem.

\end{proof}

The previous theorem yields the following corollary.

\begin{cor}\label{cor95}
The following inclusion is strict.
$${\rm {XZ}}-\mathcal {SL} \subsetneq \mathcal{WL}.$$
\end{cor}

\subsection{The proof of the first set inequality in our main result (Theorem \ref{thm13})}\label{ssec92}

For singular integrals  of convolution type $S_\varphi$, the reverse inclusion in the first set equality in Theorem \ref{thm13}
takes the following form. According to Corollary \ref{cor84}
, for $\varphi \in H^2 (\mathbb C^n, d\mu),$ the implication $\rm{(i)}\Rightarrow \rm{(ii)}$ holds, where
\begin{enumerate}
\item[{(i)}]
there exist two positive constants $C$ and $\beta>2n$ such that for all $z, w\in \mathbb C^n,$ we have
$$e^{-\frac {|w-z|^2}2}\left \vert \varphi (w-\bar{z})\right \vert\leq \frac C{(1+|w-z|)^\beta}$$
and
\item[(ii)]
there exist two positive constants $C$ and $\epsilon$ such that for all $z, w\in \mathbb C^n,$ we have
$$e^{-\frac {|w-z|^2}2}\left \vert \varphi (w-\bar{z}\right \vert\leq Ce^{-\epsilon|w-z|^2}.$$
\end{enumerate}
We apply the change of variables $s+it=w-\bar{z}$ and $y=\frac {z-\bar{z}}{2i}.$ The implication ${\rm (i)}\Rightarrow {\rm (ii)}$ then amounts to the implication ${\rm (iii)}\Rightarrow {\rm (iv)},$ where
\begin{enumerate}
\item[(iii)]
there exist two positive constants $C$ and $\beta>2n$ such that for all $s, t \in \mathbb R^n,$ we have
$$\left \vert \varphi (s+it)\right \vert\leq \inf \limits_{y\in \mathbb R^n} \frac {Ce^{\frac {s^2+(t-2y)^2}2}}{(1+\sqrt {s^2+(t-2y)^2})^\beta}$$
and
\item[(iv)]
there exist two positive constants $C$ and $\epsilon<\frac 12$ such that for all $s, t\in \mathbb R^n,$ we have
$$\left \vert \varphi (s+it)\right \vert\leq C\inf \limits_{y\in \mathbb R^n}e^{\left (\frac 12-\epsilon\right )\left (s^2+(t-2y)^2\right )}.$$
\end{enumerate}
It is easy to check the assertion {\rm (iv)} is equivalent to the following assertion:
\begin{enumerate}
\item[(v)]
there exist two positive constants $C$ and $\epsilon<\frac 12$ such that for all $s, t\in \mathbb R^n,$ we have
$$\left \vert \varphi (s+it)\right \vert\leq Ce^{\left (\frac 12-\epsilon\right )s^2}.$$
\end{enumerate}
To reformulate the assertion {\rm (iv)}, we need to determine $\inf \limits_{y\in \mathbb R^n} \frac {e^{\frac {s^2+(t-2y)^2}2}}{\left (1+\frac {s^2+(t-2y)^2}2\right )^{\frac \beta 2}}.$ We shall use the following elementary lemma.

\begin{lem}\label{lem96}
Let $\tau$ be a positive number. We denote by $h$ the following real function of one real variable $x:$
$$h(x):=\frac {e^{\frac {2x}\beta}}{1+x} \quad (x\in (\tau, \infty)).$$
Then $\inf \limits_{x\in (\tau, \infty)} h(x)=\frac {e^{\frac {2\tau}\beta}}{1+\tau}$ if $\tau>\frac \beta 2-1$ and $\inf \limits_{x\in (\tau, \infty)} h(x)\leq \frac {2e^{\frac 2\beta\left (\frac \beta 2-1\right )}}\beta$ if $\tau\leq \frac \beta 2-1.$
\end{lem}

We take $\tau=\frac {s^2}2$ and $x=\frac {s^2+(t-2y)^2}2\in \left (\frac {s^2}2, \infty\right ).$ It follows from the previous lemma that $$\inf \limits_{y\in \mathbb R^n} \frac {e^{\frac {s^2+(t-2y)^2}2}}{\left (1+\frac {s^2+(t-2y)^2}2\right )^{\frac \beta 2}} \simeq \frac {e^{\frac {s^2}2}}{(1+|s|)^\beta}.$$ 
So the implication ${\rm (iii)}\Rightarrow {\rm (iv)}$ amounts to the implication ${\rm (vi)}\Rightarrow {\rm (v)},$ where  
\begin{enumerate}
\item[(vi)]
there exist two positive constants $C$ and $\beta>2n$ such that for all $s, t \in \mathbb R^n,$ we have
$$\left \vert \varphi (s+it)\right \vert\leq \frac {Ce^{\frac {s^2}2}}{(1+|s|)^\beta}.$$
\end{enumerate}

We now prove the following theorem.

\begin{thm}\label{thm97}
For $n=1,$ there exists a function $\varphi \in H^2 (\mathbb C, d\mu)$ which satisfies the assertion {\rm (vi)} but does not satisfy the assertion {\rm (v)}.
\end{thm}

\begin{proof}
Let $\beta$ be an integer greater than $2.$ We consider the entire function $\varphi (z)=e^{\frac {z^2}2}\left [\frac {\sin z}z \right ]^\beta.$

We first show that $\varphi \in H^2 (\mathbb C, d\mu).$ We have
$$\int_{\mathbb C} |\varphi (z)|^2d\mu (z)=I_1+I_2,$$
where $I_1:=\int_{s^2+t^2\leq 1} |\varphi (s+it)|^2 e^{-s^2-t^2}dsdt$ and $I_2:=\int_{s^2+t^2 >1} |\varphi (s+it)|^2 e^{-s^2-t^2}dsdt.$ We must prove that $I_1<\infty$ and $I_2<\infty.$

For $I_1,$ since there exists a positive constant $C$ such that $\left \vert \frac {\sin (s+it)}{s+it}\right \vert\leq C$ whenever $s^2+t^2\leq 1,$ we obtain:
$$I_1\leq C^{2\beta}\int_{s^2+t^2\leq 1} e^{s^2-t^2}e^{-s^2-t^2}dsdt\leq C^{2\beta}\int_{s^2+t^2\leq 1} ds dt<\infty.$$ 
For $I_2,$ we have $|\sin (s+it)|=\frac 12\left \vert e^{i(s+it)}-e^{-i(s+it)}\right \vert \leq \frac 12 (e^t+e^{-t})\leq e^{|t|};$ Since
$$I_2=\int_{s^2+t^2 >1} e^{s^2-t^2}\left \vert \frac {\sin (s+it)}{s+it}\right \vert^{2\beta} e^{-s^2-t^2}dsdt,$$
we get:
$$I_2\leq \int_{s^2+t^2 >1} \frac {e^{-2t^2+2\beta |t|}}{(s^2+t^2)^\beta}dsdt.$$
Next, it is easy to check that  there exists a positive constant $C=C(\beta)$ such that for all $t\in \mathbb R,$ we have $e^{-2t^2+2\beta |t|}\leq C.$ So
$$I_2\leq C\int_{s^2+t^2 >1} \frac 1{(s^2+t^2)^\beta}dsdt<\infty$$
because $\beta>2.$

We next prove the assertion {\rm (vi)}. On the one hand, we have:
\begin{eqnarray*}
\left \vert \varphi (s+it)\right \vert&=e^{\frac {s^2-t^2}2}\left \vert \frac {\sin (s+it)}{s+it}\right \vert^\beta\\
&\leq Ce^{\frac {s^2}2}\leq \frac {C'e^{\frac {s^2}2}}{(1+|s|)^\beta}
\end{eqnarray*}
whenever $s^2+t^2\leq 1.$ On the other hand, we have the implication:
$$s^2+t^2 >1 \Rightarrow |s|> \frac 1{\sqrt 2} \hskip 2truemm {\rm or} \left (|s|\leq \frac 1{\sqrt 2} \hskip 2truemm {\rm and} \hskip 2truemm |t|> \frac 1{\sqrt 2} \right ).$$
When $|s|> \frac 1{\sqrt 2},$ we obtain:
$$\left \vert \varphi (s+it)\right \vert \leq e^{\frac {s^2}2}\frac {e^{-\frac {t^2}2+\beta |t|}}{|s|^\beta}\leq C_\beta \frac {e^{\frac {s^2}2}}{(1+|s|)^\beta}.$$
When $|s|\leq \frac 1{\sqrt 2} \hskip 2truemm {\rm and} \hskip 2truemm |t|> \frac 1{\sqrt 2},$ we get:
$$\left \vert \varphi (s+it)\right \vert \leq c_\beta e^{\frac {s^2}2}\frac {e^{-\frac {t^2}2+\beta |t|}}{(1+|s|)^\beta}\leq C_\beta \frac {e^{\frac {s^2}2}}{(1+|s|)^\beta}.$$

We finally show that $\varphi$ does not satisfy the assertion {\rm (v)}. We take $t=0$ and $s=\frac {(2n+1)\pi}2 \hskip 2truemm (n\in \mathbb N);$ then:
$$|\varphi (s)|=e^{\frac {s^2}2}\left \vert \frac {\sin (s)}{s}\right \vert^\beta =\frac {e^{\frac {s^2}2}}{|s|^\beta}.$$
If the assertion {\rm (v)} were true, there would exist two positive constants $C$ and $\epsilon$ such that
$$\frac {e^{\frac {s^2}2}}{|s|^\beta}\leq Ce^{\left (\frac 12-\epsilon \right )s^2}$$
for all $s=\frac {(2n+1)\pi}2 \hskip 2truemm (n\in \mathbb N);$ this amounts to the estimate:
$$e^{\epsilon s^2}\leq C|s|^\beta.$$
Letting $n$ tend to $\infty,$ we conclude that this estimate is false. 
\end{proof}

We deduce the following corollary.

\begin{cor}\label{cor98}
The following inclusion is strict:
$$\mathcal {SL}\subsetneq {\rm XZ-}\mathcal {SL}.$$

\end{cor}

\section{Final questions}\label{sec10}
At the end of this article, we list the following questions we were not able to solve:
\begin{enumerate}
\item
Are the Toeplitz operators studied in  Theorem \ref{thm514}
strongly localized? This is the case for the analogous operators on the Bergman space on the unit disc \cite{Z22}.
\item
Is the operator $T_\gamma$ defined in Subsection \ref{ssec62} 
weakly localized?
\end{enumerate}

\section*{Acknowledgement}
The authors wish to thank Aline Bonami for the simplification of the proof of Proposition \ref{pro61}
 and for suggesting the proof of Theorem \ref{thm94}.

\end{document}